\documentclass[12pt,a4paper]{article}
\usepackage{amsmath}
\usepackage{amssymb}
\usepackage{amsthm}
\usepackage{amsfonts}
\usepackage{latexsym}

\theoremstyle{plain}
\newtheorem{thm}{Theorem}[section]
\newtheorem{cor}{Corollary}[section]
\newtheorem{lem}{Lemma}[section]
\newtheorem{prop}{Proposition}[section]

\theoremstyle{remark}

\theoremstyle{definition}
\newtheorem{defn}{Definition}[section]
\newtheorem{notn}{Notation}[section]

\newtheorem{rem}{Remark}[section]
\newtheorem{ba}{Basic Assumption}[section]

\title{Polarized orbifolds associated to quantized Hamiltonian torus actions}
\author{Roberto Paoletti
\footnote{\noindent{\bf Address:}
Dipartimento di Matematica e Applicazioni, Universit\`a degli Studi
di Milano-Bicocca, Via R. Cozzi 55, 20125 Milano,
Italy; {\bf e-mail}: roberto.paoletti@unimib.it }}
\date{}

\begin{document}
\maketitle

\begin{abstract}
Suppose given an holomorphic and Hamiltonian action of a compact torus $T$ on
a polarized Hodge manifold $M$. Assume that the action lifts to the quantizing line bundle, 
so that there is an induced unitary representation of $T$ on
the associated Hardy space. If in addition
the moment map is nowhere zero, for each weight $\boldsymbol{\nu}$ 
the $\boldsymbol{\nu}$-th isotypical component in the Hardy space of the
polarization is finite-dimensional. 
Assuming that the moment map is 
transverse to the ray through $\boldsymbol{\nu}$, we give a gometric interpretation of
the isotypical components associated to the weights $k\,\boldsymbol{\nu}$, $k\rightarrow
+\infty$, in terms
of certain polarized orbifolds associated to the Hamiltonian action and the weight. 
These orbifolds are generally not reductions of $M$ in the usual
sense, but arise rather as quotients of certain loci in the unit circle bundle of the
polarization; this construction generalizes the one of weighted projective spaces
as quotients of the unit sphere, viewed as the domain of the Hopf map.
\end{abstract}

\section{Introduction}

Let $M$ be a $d$-dimensional connected complex projective manifold, 
with complex structure $J$.
Let
$(A,h)$ be a positive holomorphic line bundle on $(M,J)$; 
the curvature of the unique covariant derivative on $A$ compatible with both
the Hermitian metric $h$ and the complex structures has the form 
$\Theta=-2\pi\,\imath\,\omega$,
where $\omega$ is a K\"{a}hler form on $(M,J)$. 
Let $\mathrm{d}V_M:=\omega^{\wedge d}/d!$ be the associated volume form on $M$.

Let $A^\vee$ be the dual line bundle of $A$, endowed with the dual Hermitian metric $h^\vee$.
As is well-known, positivity of $(A,h)$ is equivalent to the unit disc
bundle $D\subset A^\vee$ being a strictly pseudoconvex domain \cite{grauert}.
We shall denote by $X:=\partial D\subset A^\vee$ the unit circle bundle of $h^\vee$, and by
$\alpha\in \Omega^1(X)$ the (normalized) connection 1-form on $X$. Thus,
$X$ is a principal $S^1$-bundle on $M$, with the structure $S^1$-action
$\rho^{X}:S^1\times X\rightarrow X$ given by 
clockwise fiber rotation. If $\pi:X\rightarrow M$ is the bundle projection, and
$-\partial_\theta\in \mathfrak{X}(X)$ is the generator of $\rho^{X}$, then 
\begin{equation}
\label{eqn:key properties alpha}
\mathrm{d}\alpha=2\,\pi^*(\omega),\quad \alpha (\partial_\theta)=1.
\end{equation}
Let $\mathrm{d}V_X:=(\alpha/2\pi)\wedge \,\pi^*(\mathrm{d}V_M)$ be the associated volume form on $X$

Then $\alpha$ is a contact form on $X$, and $X$ is a CR manifold, with CR structure
supported by the horizontal tangent bundle
\begin{equation}
\label{eqn:horizontal tangent bundle}
Hor(X):=\ker(\alpha)\subset TX.
\end{equation}

Let $H(X)\subseteq L^2(X)$ denote the Hardy space of $X$. 
Since $\rho^X$ preserves $\alpha$ and the CR structure, it induces a unitary
representation $\hat{\rho}^X$ of $S^1$ on $H(X)$, given by
$$
\hat{\rho}^X_{e^{\imath\,\vartheta}}(s)(x):=
s\left(\rho^{X}_{e^{-\imath\,\vartheta}}(x)\right)=
s\left(e^{\imath\,\vartheta}\,x\right)
\quad \left(x\in X,\, e^{\imath\,\vartheta}\in S^1,\,s\in H(X)\right).
$$
The induced isotypical decomposition is the Hilbert space direct sum
\begin{equation}
\label{eqn:isotypical Hardy}
H(X)=\bigoplus_{k=0}^{+\infty}H(X)_k,
\end{equation}
where 
$$
H(X)_k:=\left\{s\in H(X)\,:\,s\left(e^{\imath\,\theta}\,x\right)=
e^{\imath\, k\,\theta}\,s(x)\,\quad\,\forall\,x\in X,\,e^{\imath\,\theta}\in 
S^1\right\}.
$$
It is well-known that there are 
natural unitary isomorphisms
$H(X)_k\cong H^0(M,A^{\otimes k})$, the latter being the space of global
holomorphic sections of $A^{\otimes k}$.

Furthermore, let $T\cong (S^1)^r$ be an $r$-dimensional compact torus,
with Lie algebra and coalgebra $\mathfrak{t}$ and $\mathfrak{t}^\vee$,
respectively. 
We shall equivariantly identify 
$\mathfrak{t}\cong \mathfrak{t}^\vee\cong \imath\,\mathbb{R}^r$. 
Suppose given an 
Hamiltonian and holomorphic action 
$\mu^M:T\times M\rightarrow M$ of $T$ on the K\"{a}hler manifold
$(M,J,2\,\omega)$. Let 
$\Phi:M\rightarrow \mathfrak{t}^\vee\cong \imath\,\mathbb{R}^r$
be the moment map.

It is standard that 
$\mu^M$ and $\Phi$ generate an infinitesimal contact and CR
action of $\mathfrak{t}$ on $X$, so defined \cite{ko}. If $\xi\in \mathfrak{t}$, let
$\xi_M\in \mathfrak{X}(M)$ be the Hamiltonian vector field induced by $\xi$ on
$M$, and define a vector field $\xi_X\in \mathfrak{X}(X)$ by setting

\begin{equation}
\label{eqn:xiX}
\xi_X:=\xi_M^\sharp-\langle\Phi\circ \pi,\xi\rangle\,\partial_\theta\in \mathfrak{X}(X);
\end{equation}
here $V^\sharp\in \mathfrak{X}(X)$ denotes the horizontal lift to $X$ of a vector field
$V\in \mathfrak{X}(M)$, with respect to $\alpha$.
The $\xi_X$'s are commuting contact vector fields on $X$, whose flow preserves the CR structure,
and the map $\xi\mapsto \xi_X$ is a morphism of 
Lie algebras
$\mathfrak{t}\rightarrow \mathfrak{X}(X)$.

Let us make the stronger hypothesis that $\mu^M$ lifts to an actual
contact and CR action of $T$ on $X$, 
$\mu^X:T\times X\rightarrow X$,
and that
$\mathrm{d}\mu^X(\xi)=\xi_X$ for any $\xi\in \mathfrak{t}$.
Then $\mu^X$ determines a unitary representation $\hat{\mu}^X$
of $T$ on $H(X)$,
given by
\begin{equation}
\label{eqn:unitary action mu}
\hat{\mu}^X_{\mathbf{t}}(s)(x):=
s\left(\mu^{X}_{\mathbf{t}^{-1}}(x)\right)
\quad \left(x\in X,\, \mathbf{t}\in T,\,s\in H(X)\right).
\end{equation}
By the Peter-Weyl Theorem \cite{st}, $\hat{\mu}^X$ induces a unitary and equivariant splitting of
$H(X)$ into isotypical components. 

Let us regard
any $\boldsymbol{\nu}\in \mathbb{Z}^r$ 
as an integral weight on $T$, associated to
the character 
$$
\chi_{\boldsymbol{\nu}}(\mathbf{t}):=\mathbf{t}^{\boldsymbol{\nu}},
$$
where for $\mathbf{t}=(t_1,\ldots,t_r)\in T$ we set 
$\mathbf{t}^{\boldsymbol{\nu}}:=\prod_{j=1}^r t_j^{\nu_j}$.
For any $\boldsymbol{\nu}\in \mathbb{Z}^r$ , let us consider
the $\boldsymbol{\nu}$-th isotypical component
$$
H(X)^{\hat{\mu}}_{\boldsymbol{\nu}}:=
\left\{
s\in H(X)\,:\,\hat{\mu}_{\mathbf{t}}(s)=
\chi_{\boldsymbol{\nu}}(\mathbf{t})\cdot s\quad\forall\,\mathbf{t}\in T
\right\}.
$$
Then we have an equivariant Hilbert space direct sum
\begin{equation}
\label{ref:HTdecomposition}
H(X)=\bigoplus_{\boldsymbol{\nu}\in \mathbb{Z}^r}H(X)^{\hat{\mu}}_{\boldsymbol{\nu}}.
\end{equation}

In the special case where $T=S^1$, $\mu^M$ is trivial, and $\Phi=\imath$,
$\imath\in \mathfrak{t}$ is mapped to $-\partial_\theta$, and so $\mu^X=\rho^X$; hence
(\ref{ref:HTdecomposition}) reduces to (\ref{eqn:isotypical Hardy}), that is,
$H(X)_k=H(X)^{\hat{\rho}}_k$ with $k$ in place of $\boldsymbol{\nu}$.

In general, it may happen that
$H(X)^{\hat{\mu}}_{\boldsymbol{\nu}}\cap H(X)_k\neq (0)$ for several $k$'s,
so that $H(X)^{\hat{\mu}}_{\boldsymbol{\nu}}$ does not correspond to a space of
holomorphic sections of some power of $A$.
Furthermore, $H(X)^{\hat{\mu}}_{\boldsymbol{\nu}}$ may be infinite-dimensional.
The latter circumstance does not occur, however, if $\mathbf{0}\not\in \Phi(M)$
(see \S 2 of \cite{pao-ijm}). 
We shall make the following Basic Assumption (henceforth referred to
as BA): 

\begin{ba}
\label{ba:transv}
$\Phi$ and $\boldsymbol{\nu}$ satisfy the following properties:

\begin{enumerate}
\item $\boldsymbol{\nu}\neq \mathbf{0}$ is coprime, that is,
$\mathrm{l.c.d.}(\nu_1,\ldots,\nu_r)=1$;

\item $\Phi$ is nowhere vanishing, that is, $\mathbf{0}\not\in \Phi(M)$;

\item $\Phi$ is transverse to the ray $\mathbb{R}_+\cdot \imath\,\boldsymbol{\nu}$, and
$
M_{\boldsymbol{\nu}}:=\Phi^{-1}(\mathbb{R}_+\cdot \imath\,\boldsymbol{\nu})\neq \emptyset$.
\end{enumerate}
If the previous properties are satisfied, then 
$\mu^X$ is generically locally free \cite{pao-ijm}; perhaps after replacing $T$
with its quotient
by a finite subgroup, we may and will assume without loss of generality 
that $\mu^X$ 
is generically free.
\end{ba}

Let us assume that BA holds. Then
$H(X)^{\hat{\mu}}_{k\,\boldsymbol{\nu}}=(0)$ for all $k\le 0$
(\S 2 of \cite{pao-ijm}).
We are interested in the sequence of spaces of finite-dimensional vector spaces
$\big(H(X)^{\hat{\mu}}_{k\,\boldsymbol{\nu}}\big)_{k=1}^{+\infty}$ associated to the weights on
the ray $\mathbb{R}_+\cdot \,\boldsymbol{\nu}$.
The corresponding
\lq equivariant Szeg\"{o} projectors\rq\,
$\Pi^{\hat{\mu}}_{k\,\boldsymbol{\nu}}:L^2(X)\rightarrow H(X)^{\hat{\mu}}_{k\,\boldsymbol{\nu}}$
are smoothing operators (that is, they have $\mathcal{C}^\infty$ integral kernels).
Furthemore, $M_{\boldsymbol{\nu}}\subseteq M$ is a 
$T$-invariant 
coisotropic connected compact submanifold
of real codimension $r-1$ \cite{pao-ijm}. The local and global 
asymptotics 
for $k\rightarrow +\infty$ of 
the integral kernels $\Pi^{\hat{\mu}}_{k\,\boldsymbol{\nu}}$ and their concentration
behaviour along $M_{\boldsymbol{\nu}}$ were studied in \cite{pao-ijm}, \cite{pao-lower},
and related variants in the presense of additional symmetries where investigated in \cite{cam}.

Our present aim is to clarify the geometric significance of the sequence
$\big(H(X)^{\hat{\mu}}_{k\,\boldsymbol{\nu}}\big)_{k=1}^{+\infty}$, generalizing the
interpretation of the sequence $\big(H(X)_k\big)$ in terms of the spaces
$H^0(M,A^{\otimes k})$. We shall prove the following:

\begin{thm}
\label{thm:main}
Assume BA holds. Then there exists a $(d+1-r)$-dimensional connected compact complex orbifold 
$N_{\boldsymbol{\nu}}$, and a positive holomorphic orbifold line bundle $B_{\boldsymbol{\nu}}$
on $N_{\boldsymbol{\nu}}$, naturally constructed from $A$, $\boldsymbol{\nu}$ and $\Phi$,
such that the following holds:

\begin{enumerate}
\item for $k\ge 1$, there is a natural injection 
$\delta_k:H(X)^{\hat{\mu}}_{k\,\boldsymbol{\nu}}\hookrightarrow 
H^0\big(N_{\boldsymbol{\nu}},B_{\boldsymbol{\nu}}^{\otimes k}\big)$;
\item $\delta_k$ is an isomorphism if $k\gg 0$.
\end{enumerate}
\end{thm}

\begin{cor}
\label{cor:hol euler char}
If $k\gg 0$, 
$$
\dim H(X)^{\hat{\mu}}_{k\,\boldsymbol{\nu}} =\chi \big(N_{\boldsymbol{\nu}},B_{\boldsymbol{\nu}}^{\otimes k}\big).
$$
\end{cor}

Obviously with no pretense of exhaustiveness, discussions of orbifolds and orbifold line bundles
(also known as $V$-manifolds and line $V$-bundles) 
can be found in \cite{sat1}, \cite{sat2}, \cite{baily},
\cite{k}, \cite{alr}, \cite{bg}; specific treatments of
Hamiltonian actions on symplectic orbifolds can be found
in \cite{lt} and \cite{ms}.

The geometric significance of the Theorem lies in the
relation between the polarized orbifold 
$(N_{\boldsymbol{\nu}},B_{\boldsymbol{\nu}})$
and the \lq prequantum data\rq\, $(A,\,\Phi,\,\boldsymbol{\nu})$. It is therefore
in order to outline how the former is constructed from the latter.
The following
statements will be clarified and proved in \S \ref{sctn:preliminaries}.

Let $\tilde{T}\cong (\mathbb{C}^*)^r$ be the complexification of $T$. Then $\mu^X$ extends 
to an holomorphic line bundle action 
$\tilde{\mu}^{A^\vee}:\tilde{T}\times A^\vee\rightarrow A^\vee$. 
Let $A^\vee_0$ be the complement of the zero section in $A^\vee$, and
let $A^\vee_{\boldsymbol{\nu}}\subset A^\vee_0$ be the inverse image of $M_{\boldsymbol{\nu}}$.
Let $\tilde{A}^\vee_{\boldsymbol{\nu}}:=\tilde{T}\cdot A^\vee_{\boldsymbol{\nu}}$
be its saturation under $\tilde{\mu}^{A^\vee}$. 

Then $\tilde{\mu}^{A^\vee}$ is proper and locally free on $\tilde{A}^\vee_{\boldsymbol{\nu}}$, and 
$N_{\boldsymbol{\nu}}=\tilde{A}^\vee_{\boldsymbol{\nu}}/\tilde{T}$.
Thus the projection
$p_{\boldsymbol{\nu}}:\tilde{A}^\vee_{\boldsymbol{\nu}}\rightarrow N_{\boldsymbol{\nu}}$ 
is a principal $V$-bundle
with structure group $\tilde{T}$ over $N_{\boldsymbol{\nu}}$ \cite{sat2}.

Furthermore, $\chi_{\boldsymbol{\nu}}:T\rightarrow S^1$
extends to a character $\tilde{\chi}_{\boldsymbol{\nu}}:\tilde{T}\rightarrow \mathbb{C}^*$; 
the datum of $p_{\boldsymbol{\nu}}$ and $\tilde{\chi}_{\boldsymbol{\nu}}$ determines the
orbifold line bundle $B_{\boldsymbol{\nu}}$.
Similarly, $B_{\boldsymbol{\nu}}^{\otimes k}$ (or $B_{k\,\boldsymbol{\nu}}$)
denotes the orbifold line bundle associated to 
$p_{\boldsymbol{\nu}}$ and $\tilde{\chi}_{k\,\boldsymbol{\nu}}=
\tilde{\chi}_{\boldsymbol{\nu}}^{k}$.

We can give the following alternative 
algebro-geometric
characterization of $\tilde{A}^{\vee}_{\boldsymbol{\nu}}$.
Let $\boldsymbol{\nu}^\perp\subset \mathbb{R}^r$ be the orthocomplement of
$\boldsymbol{\nu}$ with respect to the standard scalar product, and consider
the (Abelian) subalgebra
$\imath\,\boldsymbol{\nu}^\perp\leqslant \mathfrak{t}$. Let 
$T^{r-1}_{\boldsymbol{\nu}^\perp}\leqslant T$ be the corresponding subtorus,
$\tilde{T}^{r-1}_{\boldsymbol{\nu}^\perp}\leqslant \tilde{T}$ be its
complexification.
The restriction of $\tilde{•\mu}^M$ to $\tilde{T}^{r-1}_{\boldsymbol{\nu}^\perp}$ is 
an holomorphic action $\tilde{•\gamma}^M$ of
$\tilde{T}^{r-1}_{\boldsymbol{\nu}^\perp}$
on $(M,\,J)$, with a built-in complex linearization
$\tilde{\gamma}^{A^\vee}:\tilde{T}^{r-1}_{\boldsymbol{\nu}^\perp}\times A^\vee\rightarrow A^\vee$. 
Let $\tilde{M}_{\boldsymbol{\nu}}\subseteq M$ be the locus of (semi)stable points
of $\tilde{\gamma}^M$; then $\tilde{A}^{\vee}_{\boldsymbol{\nu}}$ is the inverse image of
$\tilde{M}_{\boldsymbol{\nu}}$ in $A^\vee_0$.

Up to a natural isomorphism, an alternative description of $N_{\boldsymbol{\nu}}$ is as follows. 
Let $X_{\boldsymbol{\nu}}:=\pi^{-1}(M_{\boldsymbol{\nu}})$.
Then $T$ acts locally freely on $X_{\boldsymbol{\nu}}$, and 
$N_{\boldsymbol{\nu}}\cong X_{\boldsymbol{\nu}}/T$. 
This description is instrumental in describing the positivity of
$B_{\boldsymbol{\nu}}$ and the K\"{a}hler structure
of $N_{\boldsymbol{\nu}}$.

When $r=1$,
$M=\mathbb{P}^d$, and $A$ is the hyperplane line bundle with the 
standard metric, we have
$X_{\boldsymbol{\nu}}=X=S^{2d+1}$; thus 
the previous construction generalizes the one of weighted projective spaces
(see also the discussions in \cite{pao-lower} and \cite{pao-u2}).

\section{Preliminaries}

\label{sctn:preliminaries}

This section is devoted to a closer description of the geometric setting,
and to the statement and proof of a series of geometric results that will
combine into the proof of Theorem \ref{thm:main}. 

\begin{notn} We shall adopt the following notation and conventions.

\begin{enumerate}
\item If a Lie group $G$ with Lie algebra $\mathfrak{g}$
acts smoothly on a manifold $R$, for any $\xi \in \mathfrak{g}$
we shall denote by $\xi_R\in \mathfrak{X}(R)$ the vector field
on $R$ generated by $\xi$. 
\item If $r\in R$ and $\mathfrak{l}\subseteq \mathfrak{g}$ is a vector subspace, 
we shall set
$$
\mathfrak{l}_R(r):=\big\{\xi_R(r)\,:\,\xi\in \mathfrak{l}\big\}\subseteq T_rR.
$$
\item Given an isomorphism $T\cong (S^1)^r$, we have $\mathfrak{t}\cong\imath\,\mathbb{R}^r$. 
If we identify
the Lie algebra $\tilde{\mathfrak{t}}$ of 
$\tilde{T}\cong \mathbb{R}_+^r\times T$ with 
$\mathbb{C}^r\cong \mathbb{R}^r\oplus\imath\,\mathbb{R}^r$, 
$\mathfrak{t}$ corresponds to the imaginary summand
$\imath\,\mathbb{R}^r$.
For $\mathbf{x}=\begin{pmatrix}
x_1&\cdots&x_r
\end{pmatrix}\in \mathbb{R}^r$, 
we have $e^{\mathbf{x}}:=\begin{pmatrix}
e^{x_1}&\cdots&e^{x_r}
\end{pmatrix}\in \mathbb{R}_+^r\leqslant \tilde{T}^r$,
while $e^{\imath\,\mathbf{x}}:=\begin{pmatrix}
e^{\imath\,x_1}&\cdots&e^{\imath\,x_r}
\end{pmatrix}\in T^r$.
\item We shall equivariantly identify $\mathfrak{t}\cong \mathfrak{t}^\vee$, and view $\Phi$
as $\mathfrak{t}$-valued.
\item If $V$ is any Euclidean vector space and $\epsilon>0$, $V(\epsilon)\subset V$ will denote the
open ball in $V$ centered at the origin and of radius $\epsilon$.
\item $g(\cdot,\cdot):=\omega\big(\cdot,J(\cdot)\big)$ is the Riemannian metric
associated to $\omega$.

\item $J'$ is the complex structure of $A^\vee$.

\item The superscript $\sharp$ will denote horizontal lifts from $M$ to
either $X$ or $A^\vee$, according to the context, and will be applied to
both tangent vectors and vector subspaces of tangent spaces.
\item $\pi:X\rightarrow M$ and $\pi':A^\vee_0\rightarrow M$ are the projections.

\item If $\beta^Z:G\times Z\rightarrow Z$ is an action of the group $G$
on the set $Z$, and if $S\subseteq Z$ is $G$-invariant, we shall often denote by
$\beta^S:G\times S\rightarrow S$ the restricted action. Thus, for example,
$\tilde{T}$ acts on $A^\vee$ by $\tilde{\mu}^{A^\vee}$, on 
$A^\vee_0\subset A^\vee$ by $\tilde{\mu}^{A^\vee_0}$, on
$\tilde{A}^\vee_{\boldsymbol{\nu}}\subseteq A^\vee_0$ by 
$\tilde{\mu}^{\tilde{A}^\vee_{\boldsymbol{\nu}}}$.

\end{enumerate}
\end{notn}

\subsection{The locus $M_{\boldsymbol{\nu}}\subseteq M$}
\label{sctn:Mnu}
Let $\gamma^M:T^{r-1}_{\boldsymbol{\nu}^\perp}\times M\rightarrow M$ be the
action induced by restriction of $\mu^M$. Then $\gamma^M$ is Hamiltonian with respect
to $2\,\omega$, and its moment map 
$\Phi_{\boldsymbol{\nu}^\perp}:M\rightarrow \imath\,\boldsymbol{\nu}^\perp$
is the composition of $\Phi$ with the orthogonal projection 
$\mathfrak{t}\rightarrow \imath\,\boldsymbol{\nu}^\perp$.
Assuming BA, we can draw the following conclusions:
\begin{enumerate}
\item $\mathbf{0}\in \imath\,\boldsymbol{\nu}^\perp$ 
is a regular value of $\Phi_{\boldsymbol{\nu}^\perp}$;
\item $M_{\boldsymbol{\nu}}=\Phi_{\boldsymbol{\nu}^\perp}^{-1}(\mathbf{0})$ is a compact and
connected coisotropic submanifold of $M$, of (real) codimension $r-1$;
\item $\gamma^M$ is locally free along $M_{\boldsymbol{\nu}^\perp}$, that is,
$$
\dim \big(\imath\,\boldsymbol{\nu}^\perp\big)_M(m)=r-1\quad \forall\, m\in M_{\boldsymbol{\nu}^\perp};
$$
\item for every $m\in M_{\boldsymbol{\nu}}$, we have 
$$
T_mM_{\boldsymbol{\nu}}=\big(\imath\,\boldsymbol{\nu}^\perp\big)_M(m)^{\perp_{\omega_m}}
=J_m\left((\imath\,\boldsymbol{\nu}^\perp)_M(m)\right)^{\perp_{g_m}}.
$$
\end{enumerate} 

This implies the following statement. Let us define
\begin{equation}
\label{eqn:Psiepsilon}
\Psi:(\mathbf{x},m)\in(\imath\,\boldsymbol{\nu}^\perp)\times M_{\boldsymbol{\nu}}\mapsto
\tilde{\mu}_{e^{\mathbf{x}}}(m)\in M.
\end{equation}

\begin{lem}
\label{lem:localyl free mnu}
Given Basic Assumption \ref{ba:transv}, the following holds:
\begin{enumerate}
\item $\tilde{\gamma}^M$ is locally free along $M_{\boldsymbol{\nu}}$;
\item for any sufficiently small $\epsilon>0$, $\Psi$ in (\ref{eqn:Psiepsilon})
restricts to a diffeomorphism between 
$(\imath\,\boldsymbol{\nu}^\perp)(\epsilon)\times M_{\boldsymbol{\nu}}$ and an open tubular neighborhood
$U_{\epsilon}$ of $M_{\boldsymbol{\nu}}$ in $M$.
\end{enumerate}

\end{lem}

\subsection{The locus $X_{\boldsymbol{\nu}}\subseteq X$ and its saturation
$\tilde{X}_{\boldsymbol{\nu}}$ in $A^\vee$}

Let us set:
\begin{equation}
\label{defn:Xnu}
X_{\boldsymbol{\nu}}:=\pi^{-1}(M_{\boldsymbol{\nu}})\subseteq X.
\end{equation}
If $x\in X_{\boldsymbol{\nu}}$, then in view of (\ref{eqn:xiX})
\begin{equation}
\label{eqn:hor lift}
\big(\imath\,\boldsymbol{\nu}^\perp\big)_X(x)=
\big(\imath\,\boldsymbol{\nu}^\perp\big)_M(m)^\sharp .
\end{equation}

We have the following analogue of Lemma \ref{lem:localyl free mnu}. 

\begin{lem}
\label{lem:local diffeo tildeTXnu}
Given BA \ref{ba:transv}, the following holds:

\begin{enumerate}
\item $\mu^{X}$ is locally free along $X_{\boldsymbol{\nu}}$;
\item for any $x\in X_{\boldsymbol{\nu}}$, 
$$T_xX_{\boldsymbol{\nu}}\cap J'_x\big(\mathfrak{t}_{A^\vee}(x)\big)=(0);$$
\item for all suitably small $\epsilon>0$, the map 
$$
\Psi':(\mathbf{x},x)\in (\imath\,\mathfrak{t})\times X_{\boldsymbol{\nu}}\mapsto \tilde{\mu}^{A^\vee_0}_{e^{\mathbf{x}}}(x)\in 
A^\vee_0
$$
determines a diffemorphism
from $(\imath\,\mathfrak{t})(\epsilon)\times X_{\boldsymbol{\nu}}$ to a tubular neighborhood of
$X_{\boldsymbol{\nu}}$ in $A^\vee$;

\item $\tilde{\mu}^{A^\vee}$ is locally free 
along $X_{\boldsymbol{\nu}}$.
\end{enumerate}

\end{lem}

\begin{proof}
[Proof of Lemma \ref{lem:local diffeo tildeTXnu}]
That $\mu^X$ is locally free on $X_{\boldsymbol{\nu}}$ under the transversality assumption in
BA is proved in \S 2 of \cite{pao-ijm}.

By the discussion in \S \ref{sctn:Mnu},
if $x\in X_{\boldsymbol{\nu}}$ and $m=\pi(x)$ then
$$J_m\big(\imath\,\boldsymbol{\nu}^\perp\big)_M(m)^\sharp\subseteq T_xX$$ 
is the
normal space of $X_{\boldsymbol{\nu}}$ in $X$ at $x$; 
hence given (\ref{eqn:hor lift})
we have
\begin{equation}
\label{eqn:normalbundleXnu0}
T_xX_{\boldsymbol{\nu}}\cap 
J'_x\big(\imath\,\boldsymbol{\nu}^\perp\big)_X(x)=
T_xX_{\boldsymbol{\nu}}\cap 
J_m\big(\imath\,\boldsymbol{\nu}^\perp\big)_M(m)^\sharp =(0).
\end{equation}
Furthermore, by definition of $X_{\boldsymbol{\nu}}$
there exists a smooth function $\lambda_{\boldsymbol{\nu}}:M_{\boldsymbol{\nu}}\rightarrow 
\mathbb{R}_+$ such that 
\begin{equation}
\label{eqn:nux0}
(\imath\,\boldsymbol{\nu})_X(x)=(\imath\,\boldsymbol{\nu})_M(m)^\sharp-
\lambda_{\boldsymbol{\nu}}(m)\,\|\boldsymbol{\nu}\|^2\,
\left.\partial_\theta\right|_x\not\in Hor(X)_x,\quad 
\forall\,x\in X_{\boldsymbol{\nu}}.
\end{equation}
If $r$ denotes the radial coordinate along the fibers of $A^\vee$,
this implies
\begin{equation}
\label{eqn:Jnux0}
J'_x\big((\imath\,\boldsymbol{\nu})_X(x)\big)=J_m\big((\imath\,\boldsymbol{\nu})_M(m)\big)^\sharp+
\lambda_{\boldsymbol{\nu}}(m)\,\|\boldsymbol{\nu}\|^2\,
\left.\partial_r\right|_x\in T_xA^\vee \setminus T_xX.
\end{equation}
The second statement follows from (\ref{eqn:normalbundleXnu0}), (\ref{eqn:nux0}) and (\ref{eqn:Jnux0}).

The third statement is an immediate consequence of the second.

Since $X_{\boldsymbol{\nu}}$ is a $T$-invariant submanifold of $A^\vee$,
$\big(\imath\,\boldsymbol{\nu}^\perp\big)_X(x)\subseteq T_xX_{\boldsymbol{\nu}}$
for any $x\in X_{\boldsymbol{\nu}}$. Hence 
if $x\in X_{\boldsymbol{\nu}}$ and $m=\pi(x)$ then by (\ref{eqn:normalbundleXnu0})
\begin{equation}
\label{eqn:normalbundleXnu}
\big(\imath\,\boldsymbol{\nu}^\perp\big)_X(x)\cap 
J'_x\big(\imath\,\boldsymbol{\nu}^\perp\big)_X(x)=
\big(\imath\,\boldsymbol{\nu}^\perp\big)_M(m)^\sharp\cap 
J_m\big(\imath\,\boldsymbol{\nu}^\perp\big)_M(m)^\sharp =(0).
\end{equation}
Together with the first statement, this implies
$\dim_{\mathbb{C}}\big(\tilde{\mathfrak{t}}_{A^\vee}(x)\big)=r$ for any $x\in X_{\boldsymbol{\nu}}$.

\end{proof}

Let us consider the saturation 
\begin{equation}
\label{eqn:tildeXnu}
\tilde{X}_{\boldsymbol{\nu}}:=\tilde{T}\cdot X_{\boldsymbol{\nu}}\subseteq A^\vee .
\end{equation}

\begin{cor}
\label{cor:open saturation}
$\tilde{X}_{\boldsymbol{\nu}}$
is open in $A^\vee$.
\end{cor}

\begin{cor}
\label{cor:loc free action Xnu}
If Basic Assumption \ref{ba:transv} holds, then $\mu^X$ is generically free 
on $X_{\boldsymbol{\nu}}$.
\end{cor}

\begin{proof}
[Proof of Corollary \ref{cor:loc free action Xnu}]
If the general $x\in X_{\boldsymbol{\nu}}$ had non-trivial 
stabilizer in $T$, the same would hold of the general 
$\ell\in \tilde{X}_{\boldsymbol{\nu}}$; since the latter is open in
$A^\vee_0$, this contradicts 
the assumption that $\mu^{A^\vee}$ is generically free.
\end{proof}

\begin{cor}
\label{cor:Nnuprime orbifold}
If BA \ref{ba:transv} holds, then 
$N_{\boldsymbol{\nu}}':=X_{\boldsymbol{\nu}}/T$ is a compact orbifold
of real dimension $2\,(d+1-r)$,
and the projection 
\begin{equation}
\label{eqn:projectionprime}
p_{\boldsymbol{\nu}}':X_{\boldsymbol{\nu}}\rightarrow N'_{\boldsymbol{\nu}}
\end{equation}
is a principal $V$-bundle with structure group $T$.
\end{cor}

\begin{rem}
\label{rem:Bnu'}
Associated to $p_{\boldsymbol{\nu}}'$ and the character $\chi_{\boldsymbol{\nu}}$ there is
a orbifold complex line bundle $B_{\boldsymbol{\nu}}'$ on $N'_{\boldsymbol{\nu}}$.
\end{rem}

\subsection{The K\"{a}hler structure of $A^\vee_0$}

Let 
$\varrho:A_0^\vee\rightarrow \mathbb{R}$ denote the square norm function
in the Hermitian metric $h$, and set
\begin{equation}
\label{eqn:tildeomega}
\tilde{\omega}:=2\,\mathrm{d}\left(\Im \left(\partial\varrho^{1/2}\right)\right)=
2\,\imath\,\partial\overline{\partial}\left(\varrho^{1/2}\right)
.
\end{equation}
If $\pi':A^\vee_0\rightarrow M$ is the projection, then
\begin{equation}
\label{eqn:tildeomega1}
\tilde{\omega}=2\,\varrho\,^{1/2}\,{\pi'}^*(\omega)+\frac{\imath}{•2\,\varrho^{3/2}}
\,\partial\varrho\wedge \overline{\partial}\varrho. 
\end{equation}

The contact action $\mu^X:T\times X\rightarrow X$ extends to an holomorphic
unitary action $\mu^{A^\vee}:T\times A^\vee_0\rightarrow A^\vee_0$.

\begin{prop}
\label{prop:kahler form Avee}
$\tilde{\omega}$ is a $\mu^{A^\vee}$-invariant exact K\"{a}hler form on $A^\vee_0$.
\end{prop}

\begin{proof}
Since $\mu^{A^\vee}$ preserves both $\varrho$ and the complex structure,
by its definition $\tilde{\omega}$ is a $\mu^{A^\vee}$-invariant closed $(1,1)$-form.
Thus we need only prove that $\tilde{\omega}$ is non-degenerate.

The
unique connection compatible with both $h$ and the holomorphic structure
determines an invariant decomposition 
\begin{equation}
\label{eqn:generatig fileds Avee}
TA^\vee_0=Hor(A^\vee_0)\oplus Ver(A^\vee_0),
\end{equation}
where
\begin{equation}
\label{eqn:horizontal and vertical}
Hor(A^\vee_0):=\ker(\partial\varrho),\quad Ver(A^\vee_0):=\ker (\mathrm{d}\pi')\subset TA^\vee 
\end{equation}
denote the horizontal and vertical tangent bundles.
Then $Hor(A^\vee_0)$ and $Ver(A^\vee_0)$ are complex vector subbundles
of $TA^\vee_0$, and by (\ref{eqn:tildeomega1}) they are orthogonal for $\tilde{\omega}$.
Furthermore, the first summand on the right hand side of 
(\ref{eqn:tildeomega1}) is symplectic on $Hor(A^\vee_0)$ and vanishes on $Ver(A^\vee_0)$, and conversely for the second summand. Hence $\tilde{\omega}$ is non-degenerate.

\end{proof}


\begin{cor}
\label{cor:ml hamiltonian}
$\mu^{A^\vee}$ is Hamiltonian on $(A^\vee_0,\tilde{\omega})$, with moment map 
\begin{equation}
\label{eqn:tildePhi}
\tilde{\Phi}:=\varrho^{1/2}\cdot\Phi\circ \pi': A^\vee_0\rightarrow \mathfrak{t},
\end{equation}
where $\pi':A^\vee_0\rightarrow M$ is the projection.
\end{cor}

\begin{proof}
[Proof of Corollary \ref{cor:ml hamiltonian}]
Given an exact symplectic manifold $(R,\eta)$ with $\eta=-\mathrm{d}\lambda$, and 
a smooth Lie group action $\varsigma:G\times N\rightarrow N$ preserving $\lambda$,
it is well-known that $\varsigma$ is Hamiltonian, with moment map
$\Upsilon:R\rightarrow \mathfrak{g}^\vee$ determined by
the relation
$$
\upsilon^\xi:=\langle\Upsilon,\xi\rangle =\iota(\xi_R)\,\lambda\in \mathcal{C}^\infty(M).
$$
In our setting, $R=A^\vee_0$, $\varsigma=\mu^{A^\vee}$,
$\eta=\tilde{\omega}$ and, in view of (\ref{eqn:tildeomega}),
$$
\lambda=-2\,\left(\Im \left(\partial\varrho^{1/2}\right)\right)=
\imath\,\left(\partial\varrho^{1/2}-\overline{\partial}\varrho^{1/2}\right); 
$$
furthermore, for any $\xi\in \mathfrak{t}$ we have
$$
\xi_{A^\vee}=\xi_M^\sharp-\langle\Phi\circ\pi',\xi\rangle\,
\partial_\vartheta.
$$ 
Since $\xi_{M}^\sharp$ is a section of $Hor(A^\vee_0)$, it follows from
(\ref{eqn:horizontal and vertical}) that
$\iota(\xi_{M}^\sharp)\,\partial\varrho=0$. 
Furthermore, one can verify that 
$\imath(\partial_\theta)\,\partial\rho=\imath\,\rho$.
Putting this together, we conclude that $\mu^{A^\vee}$
is Hamiltonian, and furthermore the component 
$\tilde{•\phi}^\xi=\langle\tilde{\Phi},\xi\rangle$ of the moment map is 
\begin{eqnarray*}
\tilde{\phi}^\xi&=&\imath\cdot\varrho^{-1/2}\,
\iota\left(\xi_{M}^\sharp-(\varphi^\xi\circ \pi')\,\partial _\theta\right)\,\left(\partial\varrho-
\overline{\partial}\varrho\right)\\
&=&(\varphi^\xi\circ \pi')\,\varrho^{1/2}.
\end{eqnarray*}

\end{proof}

Let $\{\cdot,\cdot\}_{A^\vee_0}$ denote by Poisson brackets on 
$(A^\vee_0,\tilde{\omega})$.
Since $\mu^{A^\vee}$
is unitary, in view of (\ref{eqn:tildePhi}) 
we conclude the 
the following.

\begin{cor}
\label{cor:isotropico orbits}
$\{\tilde{\phi}^\xi,\tilde{\phi}^\eta\}$ vanishes,
$\forall\,\xi,\eta\in \mathfrak{t}$. In particular,
the orbits of $\mu^{A^\vee}$ in $A^\vee_0$ are isotropic for $\tilde{\omega}$.
\end{cor}

Therefore:

 \begin{cor}
\label{cor:totally real subspace}
For every $\ell\in A^\vee_0$, 
$\mathfrak{t}_{A^\vee}(\ell)\subseteq T_\ell A^\vee$ is totally
real, that is,
$$
\mathfrak{t}_{A^\vee}(\ell)\cap J'_\ell\big(\mathfrak{t}_{A^\vee}(\ell)\big)=(0).
$$
\end{cor}

By Proposition 1.6 and 
Theorem 1.12 in \cite{sj}, Corollary \ref{cor:ml hamiltonian} has the
following consequences. 

\begin{cor}
\label{cor:holomorphic slice and complexified stabilizer}
For every $\ell\in A^\vee_0$, the following holds:
\begin{enumerate}
\item the stabilizer $\tilde{T}_{\ell}\leqslant \tilde{T}$
of $\ell$ for $\tilde{\mu}^{A^\vee}$ is the complexification of of the stabilizer $T_\ell\leqslant T$
of $\ell$ for $\mu^{A^\vee}$;
\item there exists an holomorphic slice for $\tilde{\mu}^{A^\vee_0}$ at
$\ell$.

\end{enumerate}

\end{cor}

Let $A^\vee_{lf}\subseteq A^\vee_0$ be the open subset where $\tilde{\mu}^{A^\vee_0}$
is locally free. It follows from Proposition 1.6 of \cite{sj} and Corollary \ref{cor:isotropico orbits} above that the stabilizer in
$\tilde{T}$ of any $\ell\in A^\vee_{lf}$ is finite and contained in 
$T$.

\begin{defn}
\label{defn:pointwise proper}
Following \cite{sj}, we shall call $\tilde{\mu}^{A^\vee_0}$ \textit{proper at} $\ell\in A^\vee_0$
if for all sequences $(\ell_j)\subset A^\vee_0$ and $(\tilde{t}_j)\subset \tilde{T}$
such that $\ell_j\rightarrow \ell$ and 
$\tilde{\mu}^{A^\vee_0}_{\tilde{t}_j}(\ell_j)$ converges to some point
in $A^\vee_0$, $(\tilde{t}_j)$ is convergent in $\tilde{T}$.

\end{defn}

\begin{rem}
Let $U\subseteq A^\vee_0$ be a $\tilde{T}$-invariant open set. Then:
\begin{enumerate}
\item if $\tilde{\mu}^{A^\vee_0}$
is proper at every $\ell\in U$, then so is \textit{a fortiori} $\tilde{\mu}^{U}$;
\item if $\tilde{\mu}^{U}$ is proper at every $\ell\in U$, $\tilde{\mu}^{U}$
is (globally) proper.
\end{enumerate}

\end{rem}

\begin{cor}
\label{cor:proper action}
Given BA, the following holds:
\begin{enumerate}
\item $\tilde{\mu}^{A^\vee_0}$ is proper on $A^\vee_{lf}$ (that is, 
$\tilde{\mu}^{A^\vee_{lf}}$ is proper);
\item  $\tilde{X}_{\boldsymbol{\nu}}\subseteq A^\vee_{lf}$;
\item $\tilde{\mu}^{A^\vee_0}$ is proper on $\tilde{X}_{\boldsymbol{\nu}}$
(that is, $\tilde{\mu}^{\tilde{X}_{\boldsymbol{\nu}}}$ is proper).
\end{enumerate}

\end{cor}

\begin{proof}
[Proof of Corollary \ref{cor:proper action}]
We have remarked that the stabilizer in
$\tilde{T}$ of any $\ell\in A^\vee_{lf}$ coincides with the
stabilizer of $\ell$ in $T$, and therefore it is finite.
In view of Theorem
1.22 of \cite{sj}, $\tilde{\mu}^{A^\vee_0}$ is proper 
at any $\ell\in A^\vee_{lf}$, and therefore it is proper on
$A^\vee_{lf}$. This proves the first statement.

We know that $\mu^{X_{\boldsymbol{\nu}}}$ is locally free; in other words,
$\mu^{A^\vee_0}$ is locally
free along $X_{\boldsymbol{\nu}}$. 
Therefore, $\mu^{A^\vee_0}$ is locally
free on $\tilde{X}_{\boldsymbol{\nu}}=\tilde{T}\cdot X_{\boldsymbol{\nu}}$, 
because $\tilde{T}$ is Abelian. 
Therefore, 
by Corollary \ref{cor:holomorphic slice and complexified stabilizer}, 
the stabilizer of any $\ell\in \tilde{X}_{\boldsymbol{\nu}}$ in $\tilde{T}$
for $\tilde{\mu}^{A^\vee_0}$ is finite, since it 
coincides with the stabilizermof $\ell$ in $T$ for $\mu^{A^\vee_0}$. 
This proves the second statement.

The third statement is a straighforward consequence of the first two.
\end{proof}

The structure $S^1$-action $\rho^X$ extends to the holomorphic action 
$$
\tilde{\rho}^{A^\vee}:(z,\,\ell)\in\mathbb{C}^*\times A^\vee_0\mapsto
z^{-1}\,\ell\in A^\vee_0,
$$
whose orbits are the fibers of $A^\vee_0$ over $M$.

\begin{lem}
\label{lem:C*invariance}
$\tilde{X}_{\boldsymbol{\nu}}$ is $\tilde{\rho}^{A^\vee}$-invariant.
\end{lem}

\begin{proof}
[Proof of Lemma \ref{lem:C*invariance}]
Since $\tilde{\mu}^{A^\vee}$ and $\tilde{\rho}^{A^\vee}$ commute, it 
suffices to show that for any $x\in X_{\boldsymbol{\nu}}$
and $z\in \mathbb{C}^*$ we have 
$z\,x\in \tilde{X}_{\boldsymbol{\nu}}$.

Let us set 
$$
C_x:=\left\{z\in \mathbb{C}^*\,:\,z\,x\in \tilde{X}_{\boldsymbol{\nu}}\right\}.
$$
Then $1\in C_x$ and $C_x$ is open in $\mathbb{C}^*$ because
scalar multiplication is continuous and $\tilde{X}_{\boldsymbol{\nu}}$
is open in $A^\vee$ (Corollary \ref{cor:open saturation}).

Suppose $z_\infty\in \mathbb{C}^*$ is a limit point of
$C_x$. Then there exist $z_1,\,z_2,\ldots\in C_x$ such that 
$z_i\rightarrow z_\infty$.
By definition of $C_x$, for any $i=1,2,\ldots$ 
we have $z_i\,x\in \tilde{X}_{\boldsymbol{\nu}}$ for any 
$i=1,2,\ldots$. By definition of $\tilde{X}_{\boldsymbol{\nu}}$, therefore,
there exist $\tilde{t}_i\in \tilde{T}$
and $x_i\in X_{\boldsymbol{\nu}}$ such that
$z_i\,x=\tilde{\mu}^{A^\vee}_{\tilde{t}_i}(x_i)$. Since
$A^\vee_{lf}$ in Corollary \ref{cor:proper action} is clearly
$\mathbb{C}^*$-invariant, we have $z_\infty\,x\in A^\vee_{lf}$. Thus
$$
\tilde{\mu}^{A^\vee}_{\tilde{t}_i}(x_i)=z_i\,x\rightarrow z_\infty\,x\in A^\vee_{lf}.
$$
By Corollary \ref{cor:proper action} and the compactness of
$X_{\boldsymbol{\nu}}$, perhaps replacing $(\tilde{t}_i)$
and $(x_i)$ by subsequences, we may assume that $\tilde{t}_i\rightarrow \tilde{t}_\infty\in \tilde{T}$ and $x_i\rightarrow x_\infty\in X_{\boldsymbol{\nu}}$.
Hence $$
z_\infty\,x=\tilde{\mu}^{A^\vee}_{\tilde{t}_\infty}(x_\infty)\in \tilde{X}_{\boldsymbol{\nu}}\,\Rightarrow\,z_\infty\in C_x.
$$
We conclude that $C_x=\mathbb{C}^*$ for any $x\in X_{\boldsymbol{\nu}}$.
\end{proof}

Let us set, as in the Introduction,
$$
A^\vee_{\boldsymbol{\nu}}:=(\pi')^{-1}(M_{\boldsymbol{\nu}}),\quad
\tilde{A}^\vee_{\boldsymbol{\nu}}:=\tilde{T}\cdot A^\vee_{\boldsymbol{\nu}}.
$$

\begin{cor}
\label{cor:tildeXnu vs tildeAnu}
$\tilde{X}_{\boldsymbol{\nu}}=\tilde{A}^\vee_{\boldsymbol{\nu}}$.
\end{cor}

\begin{proof}
[Proof of Corollary \ref{cor:tildeXnu vs tildeAnu}]
Since $X_{\boldsymbol{\nu}}\subset A^\vee_{\boldsymbol{\nu}}$, clearly
$\tilde{X}_{\boldsymbol{\nu}}\subseteq\tilde{A}^\vee_{\boldsymbol{\nu}}$.
On the other hand, $\tilde{X}_{\boldsymbol{\nu}}$ is $\mathbb{C}^*$-invariant
by Lemma \ref{lem:C*invariance}
and contains $X_{\boldsymbol{\nu}}$, hence 
$\tilde{X}_{\boldsymbol{\nu}}\supseteq A_{\boldsymbol{\nu}}$. 
Since $\tilde{X}_{\boldsymbol{\nu}}$ is $\tilde{T}$-invariant,
we also have
$\tilde{X}_{\boldsymbol{\nu}}\supseteq \tilde{A}^\vee_{\boldsymbol{\nu}}$.
\end{proof}

It follows from Lemma \ref{lem:C*invariance} that
$\tilde{X}_{\boldsymbol{\nu}}$ is the inverse image of a $\tilde{T}$-invariant open set
of $M$. More precisely, let 
$$
M'_{\boldsymbol{\nu}}:=\tilde{T}\cdot M_{\boldsymbol{\nu}}\subseteq M.
$$
Since $\pi'$ is a submersion and intertwines $\tilde{\mu}^M$ and
$\tilde{\mu}^{A^\vee}$, we conclude the following:

\begin{cor}
\label{cor:Mprimenu}
$M'_{\boldsymbol{\nu}}$ is open in $M$ and 
$\tilde{X}_{\boldsymbol{\nu}}=(\pi')^{-1}(M'_{\boldsymbol{\nu}})$. 

\end{cor}

As in the Introduction, let $\tilde{M}_{\boldsymbol{\nu}}\subseteq M$ be the dense open subset
of stable points for $\gamma^M$. Obviously $\tilde{M}_{\boldsymbol{\nu}}$ is
$\tilde{\gamma}^M$-invariant (notation is as in the Introduction and \S \ref{sctn:Mnu}).

\begin{lem}
\label{lem:XnuMnuinv}
$\tilde{M}_{\boldsymbol{\nu}}=M'_{\boldsymbol{\nu}}$.
 \end{lem}

\begin{proof}
[Proof of Lemma \ref{lem:XnuMnuinv}]
Since $\mathbf{0}\in \imath\,\boldsymbol{\nu}^\perp$ 
is a regular value of $\Phi_{\boldsymbol{\nu}^\perp}$, 
$\tilde{M}_{\boldsymbol{\nu}}=
\tilde{T}^{r-1}_{\boldsymbol{\nu}^\perp}\cdot M_{\boldsymbol{\nu}}$.
Hence trivially
$\tilde{M}_{\boldsymbol{\nu}}=
\tilde{T}^{r-1}_{\boldsymbol{\nu}^\perp}\cdot M_{\boldsymbol{\nu}}\subseteq 
\tilde{T}\cdot M_{\boldsymbol{\nu}}=M'_{\boldsymbol{\nu}}$.

To prove the converse inclusion it suffices to check that $\tilde{M}_{\boldsymbol{\nu}}$
is $\tilde{T}$-invariant. For $k=1,2,\ldots$, 
let $\tilde{\hat{\mu}}^{(k)}$ be the representation of $\tilde{T}$
on $H^0(M,A^{\otimes k})$ induced by $\tilde{\mu}^{A^\vee}$, and let
$H^0(M,A^{\otimes k})^{T^{r-1}_{\boldsymbol{\nu}^\perp}}\subseteq H^0(M,A^{\otimes k})$
be the subspace of those sections that are invariant under 
$T^{r-1}_{\boldsymbol{\nu}^\perp}$ (equivalently, 
$\tilde{T}^{r-1}_{\boldsymbol{\nu}^\perp}$). Then $m\in \tilde{M}_{\boldsymbol{\nu}}$ if and
only if for some 
$k=1,2,\ldots$
there exists $\sigma\in H^0(M,A^{\otimes k})^{T^{r-1}_{\boldsymbol{\nu}^\perp}}$
such that $\sigma (m)\neq 0$. Since $\tilde{T}$ is Abelian, 
$\tilde{\hat{\mu}}^{(k)}_{\tilde{t}}(\sigma)\in 
H^0(M,A^{\otimes k})^{T^{r-1}_{\boldsymbol{\nu}^\perp}}$ for any $\tilde{t}\in \tilde{T}$;
therefore, if $m\in M$ is stable for $\gamma^M$, then so is $\tilde{\mu}^M_{\tilde{t}}(m)$,
for any $\tilde{t}\in \tilde{T}$.

\end{proof}

In the following, we shall write $\tilde{A}^\vee_{\boldsymbol{\nu}}$ for
$\tilde{X}_{\boldsymbol{\nu}}$. Since $\tilde{\mu}^{A^\vee}$ is holomorphic, 
proper, effective and
locally free on $\tilde{A}^\vee_{\boldsymbol{\nu}}$, we reach the following
conclusion.

\begin{cor}
\label{cor:complex quotient}
If BA \ref{ba:transv} holds, then 
$N_{\boldsymbol{\nu}}:=\tilde{A}^\vee_{\boldsymbol{\nu}}/\tilde{T}$ is a compact and
connected  orbifold
of complex dimension $d+1-r$,
and the projection 
\begin{equation}
\label{eqn:projection}
p_{\boldsymbol{\nu}}:\tilde{A}^\vee_{\boldsymbol{\nu}}\rightarrow N_{\boldsymbol{\nu}}
\end{equation}
is a principal $V$-bundle with structure group $\tilde{T}$.
\end{cor}

\begin{proof}
Since $\tilde{T}$ acts properly,
holomorphically and locally freely on $\tilde{X}_{\boldsymbol{\nu}}$,
$N_{\boldsymbol{\nu}}$ is a connected complex orbifold of dimension $d+1-r$.
Furthermore, by definition of $\tilde{X}_{\boldsymbol{\nu}}$,
$p_{\boldsymbol{\nu}}(X_{\boldsymbol{\nu}})=N_{\boldsymbol{\nu}}$. Hence $N_{\boldsymbol{\nu}}$ is compact.

\end{proof}

\begin{rem}
\label{rem:Bnu}
The holomorphic slices in 
Corollary \ref{cor:holomorphic slice and complexified stabilizer} provide local uniformizing charts for $N_{\boldsymbol{\nu}}$.
Associated to $p_{\boldsymbol{\nu}}$ and the character $\tilde{\chi}_{\boldsymbol{\nu}}$ there is
an holomorphic orbifold line bundle $B_{\boldsymbol{\nu}}$ on $N_{\boldsymbol{\nu}}$. 
\end{rem}

\subsection{The isomorphism between $N_{\boldsymbol{\nu}}'$ and $N_{\boldsymbol{\nu}}$}

We shall see that $N_{\boldsymbol{\nu}}'$ has a natural complex structure, and that
the pairs $(N_{\boldsymbol{\nu}}',B_{\boldsymbol{\nu}}')$ and $(N_{\boldsymbol{\nu}},B_{\boldsymbol{\nu}})$
in Corollaries \ref{cor:Nnuprime orbifold} and \ref{cor:complex quotient} are naturally isomorphic
as complex orbifolds and orbifold line bundles.

If $F\subseteq A^\vee_0$ is an holomorphic slice for
$\tilde{\mu}^{A^\vee}$ as in 
Corollary \ref{cor:holomorphic slice and complexified stabilizer},
let $J^F$ be its complex structure. Then $(F,J^F)$ is a complex
submanifold of $(A^\vee_0,J')$, and 
provides a local uniformizing chart for the complex orbifold 
$N_{\boldsymbol{\nu}}$. 

On the other hand, given $x\in X_{\boldsymbol{\nu}}$ 
let $F\subseteq X_{\boldsymbol{\nu}}$ be a slice at $x$ for the
action $\mu^{X_{\boldsymbol{\nu}}}:T\times X_{\boldsymbol{\nu}}\rightarrow X_{\boldsymbol{\nu}}$
induced by $\mu^{X}$. 
The stabilizer $T_x\leqslant T$ of $x$ in $T$
is a finite subgroup of $T$, and 
by Corollary \ref{cor:holomorphic slice and complexified stabilizer} $T_x=\tilde{T}_x$
(the stabilizer in $\tilde{T}$).

If $\epsilon>0$, let $F_{\epsilon}\subseteq F$ be the intersection of $F$ with an open ball 
centered at $x$ and radius $\epsilon$, in the K\"{a}hler metric 
on $A^\vee_0$
associated to $\tilde{\omega}$ in (\ref{eqn:tildeomega1}).

The proof of the following will be omitted.

\begin{prop}
\label{prop:slices TtildeT}
If $\epsilon>0$ is suitably small, $F_{\epsilon}$ is a slice for 
of $\tilde{\mu}^{A^\vee_0}$.
\end{prop}

Certainly $F$ is not a complex submanifold of $A^\vee_0$, and
in fact it does not contain any complex submanifold of positive dimension.
Nonetheless, there is a natural complex structure $J^F$ on it, that may be described as follows.

If $\ell\in \tilde{A}^\vee_{\boldsymbol{\nu}}$, 
the tangent space to the $\tilde{T}$-orbit of $\ell$, 
$\tilde{\mathfrak{t}}_{A^\vee_0}(\ell)\subseteq T_\ell A^\vee_0$,
is an $r$-dimensional complex subspace; 
let $S_\ell\subset T_\ell A^\vee_0$ be the orthocomplement
of $\tilde{\mathfrak{t}}_{A^\vee_0}(\ell)$ for the Riemannian metric associated to (\ref{eqn:tildeomega1}).
Thus $S_\ell$ is a complex subspace of $T_\ell A^\vee_0$, of dimension $d+1-r$,
and we have a smoothly varying direct sum decomposition 
$T_\ell A^\vee=\tilde{t}_{A^\vee}(\ell)\oplus S_\ell$. 
Globally on $\tilde{A}^\vee_{\boldsymbol{\nu}}$,
this yields a 
vector bundle decomposition $T A^\vee=\tilde{\mathfrak{t}}_{A^\vee}\oplus S$. 
Projecting along $\tilde{\mathfrak{t}}_{A^\vee}$, we obtain a morphism of vector bundles
$\Pi:T A^\vee\rightarrow S$ (on $A^\vee_{\boldsymbol{\nu}}$).

Let $F$ by any slice for 
$\tilde{\mu}^{A^\vee_0}$ in $\tilde{A}^\vee_{\boldsymbol{\nu}}$; in particular, 
by Proposition \ref{prop:slices TtildeT},
$F$ might be a slice for 
$\mu^{X_{\boldsymbol{\nu}}}$. 
At any $\ell\in F$, the restriction of $\Pi_\ell$
is an isomorphism of real vector spaces $\Pi^F_\ell:T_\ell F\rightarrow S_\ell$. 
We may define an almost complex structure $J^F$ on $F$ by
declaring $\Pi^F_\ell$ to be an isomorphism of complex vector spaces for each $\ell\in F$.
If $F$ is an holomorphic slice, $J^F$ clearly coincides with the complex structure of
$F$ as a submaifold of $A^\vee_0$.

It is clear that the same $J^F$ would be defined, if instead of $S$ one had chosen
another complementary complex subundle $S'$ to $\tilde{\mathfrak{t}}_{A^\vee_0}$.
The following characterization does not involve the choice of a specific sub-bundle.

\begin{lem}
\label{lem:unique JF}
If $\ell\in F$ and $v\in T_\ell F$, then $J^F_\ell(v)$ is uniquely determined by the
conditions: 
\begin{itemize}
\item $J^F_\ell(v)\in T_\ell L$;
\item $J^F_\ell(v)-J'_\ell(v)\in \tilde{\mathfrak{t}}_{A^\vee}(\ell)$.
\end{itemize}
\end{lem}

\begin{proof}
[Proof of Lemma \ref{lem:unique JF}]
For $v\in T_\ell A^\vee$, let $v_t\in \tilde{\mathfrak{t}}_{A^\vee}(\ell)$
and $v_s\in S_\ell$ be its components. 
As both $\tilde{\mathfrak{t}}_{A^\vee}(\ell)$
and $S_\ell$ are complex subspaces for $J'_\ell$,
$$
J'_\ell (v_s)=J'_\ell(v)_s,\quad J'_\ell (v_t)=J'_\ell(v)_t.
$$

By definition of $J^F$ if $v\in T_\ell A^\vee$ then
$$
J^F_\ell (v)_s=J'_\ell (v_s)=J'_\ell(v)_s.
$$
Hence,
$$
\big(J^F_\ell (v)-J'_\ell (v)\big)_s=J'_\ell(v)_s-J'_\ell (v)_s=0\quad\Rightarrow\quad
J^F_\ell (v)-J'_\ell (v)\in \tilde{\mathfrak{t}}_{A^\vee}(\ell).
$$

Suppose that $I^F_\ell:T_\ell F\rightarrow T_\ell F$ is another operator
such that $I^F_\ell(v)-J'_\ell(v)\in \tilde{\mathfrak{t}}_{A^\vee}(\ell)$
for every $v\in T_\ell F$.
Then (by definition of slice) $\forall\,v\in T_\ell F$ we have
$$
I^F_\ell(v)-J^F_\ell (v)=\left(I^F_\ell(v)-J'_\ell(v)\right)-\left(J^F_\ell(v)-J'_\ell(v)\right)
\in T_\ell F\cap \tilde{\mathfrak{t}}_{A^\vee}(\ell)=(0).
$$

\end{proof}

Consider two slices $F_1,\,F_2\subset \tilde{A}^\vee_{\boldsymbol{\nu}}$ for
$\tilde{\mu}^{A^\vee_0}$ such that $p_{\boldsymbol{\nu}}(F_1)\subseteq p_{\boldsymbol{\nu}}(F_2)$.
Let $\ell_j\in F_j$  be such that $p_{\boldsymbol{\nu}}(\ell_1)=p_{\boldsymbol{\nu}}(\ell_2)$.
Hence there exists $\tilde{t}\in \tilde{T}$ such that 
$\ell_2=\tilde{\mu}^{A^\vee_0}_{\tilde{t}}(\ell_1)$. Perhaps after restricting $F_1$,
we may find a unique $\mathcal{C}^\infty$ function $f:F_1\rightarrow \tilde{\mathfrak{t}}$,
such that $f(\ell_1)=\mathbf{0}$ and 
$\jmath (\ell):=\tilde{\mu}^{A^\vee_0}_{\tilde{t}\, e^{f(\ell)}}(\ell)\in F_2$, for all $\ell\in F_1$. 
Thus $\jmath:F_1\rightarrow F_2$ is an injection in the sense of Satake (\cite{sat1}, \cite{sat2}).

\begin{lem}
\label{lem:slice biholomorphism}
$\jmath:F_1\rightarrow F_2$
is $\left(J^{F_1},J^{F_2}\right)$-holomorphic.
\end{lem}

\begin{proof}
[Proof of Lemma \ref{lem:slice biholomorphism}]
By local uniqueness, it suffices to prove that 
$$
\mathrm{d}_{\ell_1}\jmath:(T_{\ell_1}F_1,J^{F_1}_{\ell_1})\rightarrow
(T_{\ell_2}F_2,J^{F_2}_{\ell_2})
$$ 
is $\mathbb{C}$-linear.
If $v\in T_{\ell_1}F$, we have
$$
\mathrm{d}_{\ell_1}f(v)\in \mathfrak{t},\quad 
\mathrm{d}_{\ell_1}f(v)_{A^\vee}\in \mathfrak{X}(A^\vee),\quad 
\mathrm{d}_{\ell_1}f(v)_{A^\vee}(\ell_2)\in \tilde{\mathfrak{t}}_{A^\vee}(\ell_2)\subseteq 
T_{\ell_2}A^\vee ,
$$
and 
\begin{equation}
\label{eqn:differential dj}
\mathrm{d}_{\ell_1} \jmath (v)=
\mathrm{d}_{\ell_1}f(v)_{A^\vee}(\ell_2)+\mathrm{d}_{\ell_1}\tilde{\mu}^{A^\vee_0}_{\tilde{t}}(v).
\end{equation}

If $w,w'\in T_{\ell_2}A^\vee$, we shall write $w\equiv w'$ to mean that 
$w-w'\in \tilde{\mathfrak{t}}_{A^\vee}(\ell_2)$. By (\ref{eqn:differential dj}),
we have
$
\mathrm{d}_{\ell_1} \jmath (v)\equiv \mathrm{d}_{\ell_1}\tilde{\mu}^{A^\vee_0}_{\tilde{t}}(v)$
for every $v\in T_{\ell_1}F_1$.
Replacing $v$ with $J^{F_1}_{\ell_1}(v)$, in view of Lemma \ref{lem:unique JF} 
we obtain
\begin{eqnarray}
\label{eqn:chain for dmuj}
\mathrm{d}_{\ell_1} \jmath \left(J^{F_1}_{\ell_1}(v)\right)&\equiv &
\mathrm{d}_{\ell_1}\tilde{\mu}^{A^\vee_0}_{\tilde{t}}\left(J^{F_1}_{\ell_1}(v)\right)\equiv
\mathrm{d}_{\ell_1}\tilde{\mu}^{A^\vee_0}_{\tilde{t}}\left(J'_{\ell_1}(v)\right)\nonumber\\
&=&J'_{\ell_2}\left(\mathrm{d}_{\ell_1}\tilde{\mu}^{A^\vee_0}_{\tilde{t}}(v)\right)\equiv
J'_{\ell_2}\left(\mathrm{d}_{\ell_1}\jmath (v)\right)\equiv 
J^{F_2}_{\ell_2}\left(\mathrm{d}_{\ell_1}\jmath (v)\right).
\end{eqnarray}

The first and the last vector in (\ref{eqn:chain for dmuj}) belong to 
$T_{\ell_2}F_2$; hence by Lemma \ref{lem:unique JF}
$\mathrm{d}_{\ell_1} \jmath \left(J^{F_1}_{\ell_1}(v)\right)=
J^{F_2}_{\ell_2}\left(\mathrm{d}_{\ell_1}\jmath (v)\right)$, for all
$v\in T_{\ell_1}F_1$.

\end{proof}

In Lemma \ref{lem:slice biholomorphism}, we may assume 
by Corollary \ref{cor:holomorphic slice and complexified stabilizer}
that $F_2$, say, 
is holomorphic; hence Lemma \ref{lem:slice biholomorphism} implies the following.

\begin{cor}
\label{cor:integrable JF}
For any slice $F\subset \tilde{A}^\vee_{\boldsymbol{\nu}}$ for $\tilde{\mu}^{A^\vee_0}$, 
$J^F$ is integrable.
\end{cor}

We may also take $F=F_1=F_2$ be a slice at $\ell\in A^\vee_0$, and consider the 
self-injections of $F$ induced by the stabilizer $T_\ell\leqslant T$ of $\ell$.

\begin{cor}
If $\ell\in A^\vee_{\boldsymbol{\nu}}$ and $F\subset A^\vee_{\boldsymbol{\nu}}$ is a slice
for $\tilde{\mu}^{A^\vee_0}$ at $\ell$, then $\tilde{T}_{\ell}$ acts holomorphically on
$(F,J^F)$.
\end{cor}

If we apply these considerations to the slices
$F\subseteq X_{\boldsymbol{\nu}}$ for
$\mu^{X_{\boldsymbol{\nu}}}$, we conclude the following.

\begin{cor}
\label{cor:Nnu'complex}
The V-manifold $N_{\boldsymbol{\nu}}'$
in Corollary \ref{cor:Nnuprime orbifold}
is complex.
\end{cor}

Since every $T$-orbit in
$X_{\boldsymbol{\nu}}$ is obviously contained in a unique $\tilde{T}$-orbit in
$\tilde{A}_{\boldsymbol{\nu}}$, there is 
a well-defined map
$$
\psi:\,T\cdot x\in N_{\boldsymbol{\nu}}'
\mapsto \tilde{T}\cdot x\in N_{\boldsymbol{\nu}}.
$$
Let $J^{N_{\boldsymbol{\nu}}'}$ and
$J^{N_{\boldsymbol{\nu}}}$ be the 
orbifold complex structures of 
$N_{\boldsymbol{\nu}}'$ and $N_{\boldsymbol{\nu}}$, respectively.

\begin{prop}
\label{prop:psi bijective}
$\psi$ is an isomorphism 
of complex orbifolds
$\left(N_{\boldsymbol{\nu}}',\,J^{N_{\boldsymbol{\nu}}'}\right)
\rightarrow \left(N_{\boldsymbol{\nu}},\,J^{N_{\boldsymbol{\nu}}}\right)$.
\end{prop}

\begin{proof}
[Proof of Proposition \ref{prop:psi bijective}]
By Corollary \ref{cor:tildeXnu vs tildeAnu},
any $\tilde{T}$-orbit in $\tilde{A}^\vee_{\boldsymbol{\nu}}$ intersects
$X_{\boldsymbol{\nu}}$; thus
$\psi$ is surjective.

To prove that $\psi$ is injective, suppose by contradiction that 
there exist $x_1,\,x_2\in X_{\boldsymbol{\nu}}$ such that
$x_2\in \tilde{T}\cdot x_1$
(i.e., $\psi (T\cdot x_1)=\psi(T\cdot x_2)$), but $x_2\not\in T\cdot x_1$
(i.e., $T\cdot x_1\neq T\cdot x_2$). 
Perhaps after replacing $x_2$ with another point in $T\cdot x_2$, we may assume that
$x_2=\tilde{\mu}^{A^\vee_0}_{e^{-\boldsymbol{\xi}}}(x_1)$ 
for some $\boldsymbol{\xi}\in \mathbb{R}^r\setminus \{\mathbf{0}\}$.
We may write uniquely $\boldsymbol{\xi}=\boldsymbol{\xi}'+a\,\boldsymbol{\nu}$,
where $\boldsymbol{\xi}'\in \boldsymbol{\nu}^\perp$ and $a\in \mathbb{R}$.
Perhaps interchanging $x_1$ and $x_2$, we may assume without loss that $a\ge 0$.

Let us set
$\boldsymbol{\eta}:=\imath\,\boldsymbol{\xi}\in \mathfrak{t}$. Considering the associated vector fields
$\boldsymbol{\xi}_{A^\vee},\,\boldsymbol{\eta}_{A^\vee}\in \mathfrak{X}(A^\vee)$ we have
$-\boldsymbol{\xi}_{A^\vee}=J'(\boldsymbol{\eta}_{A^\vee})$; hence $-\boldsymbol{\xi}_{A^\vee}$
is the gradient vector field of the Hamiltonian function $\tilde{\Phi}^{\boldsymbol{\eta}}=
\langle\tilde{\Phi},\boldsymbol{\eta}\rangle$, where $\tilde{\Phi}$ is as in (\ref{eqn:tildePhi}).

Since $x_1\in X_{\boldsymbol{\nu}}$, we have $\tilde{\Phi}(x_1)=\imath\,\lambda\,\boldsymbol{\nu}$
for some $\lambda>0$, hence 
$\tilde{\Phi}^{\boldsymbol{\eta}}(x_1)=\lambda\,a\,\|\boldsymbol{\nu}\|^2\ge 0$. 
Since $\tilde{\Phi}^{\boldsymbol{\eta}}$ is strictly increasing along its gradient flow
where the gradient is non-vanishing, 
\begin{equation}
\label{eqn:gradient vector field}
\tilde{\Phi}^{\boldsymbol{\eta}}\left(\tilde{\mu}^{A^\vee_0}_{e^{-t\,\boldsymbol{\xi}}}(x_1)\right)
> \tilde{\Phi}^{\boldsymbol{\eta}}(x_1)\ge 0\quad \forall \,t>0.
\end{equation}

On the other hand, we have
$$
\boldsymbol{\eta}_{A^\vee_0}=\boldsymbol{\eta}_M^\sharp-\tilde{\Phi}^{\boldsymbol{\eta}}\,\partial_\theta
\quad\Rightarrow\quad -\boldsymbol{\xi}_{A^\vee}=
\big(J\boldsymbol{\eta}_M)^\sharp + \tilde{\Phi}^{\boldsymbol{\eta}}\,r\,\partial_r.
$$
Here $r\,\partial_r$ is the generator of the 1-parameter group of diffeomorphisms
$\ell\mapsto e^t\,\ell$. With $\varrho$ as in (\ref{eqn:tildeomega}), 
for every $t>0$ we have
$$
-\boldsymbol{\xi}_{A^\vee}(\varrho)\left(\tilde{\mu}^{A^\vee_0}_{e^{-t\,\boldsymbol{\xi}}}(x_1)\right)
=\tilde{\Phi}^{\boldsymbol{\eta}}
\left(\tilde{\mu}^{A^\vee_0}_{e^{-t\,\boldsymbol{\xi}}}(x_1)  \right)
\,r\,\partial_r\varrho\left(\tilde{\mu}^{A^\vee_0}_{e^{-t\,\boldsymbol{\xi}}}(x_1)\right)>0.
$$
It follows that 
$\varrho\left(\tilde{\mu}^{A^\vee_0}_{e^{-t\,\boldsymbol{\xi}}}(x_1)\right)>\varrho (x_1)=1$ for $t>0$;
taking $t=1$, we conclude that $x_2\not\in X$, a contradiction.
Hence $\psi$ is a bijection. 

Let us verify that $\psi$ is a homeomorphism. 
The open sets of 
$N_{\boldsymbol{\nu}}'$ have the form $U/T$, where $U\subseteq X_{\boldsymbol{\nu}}$
is open and $T$-invariant, and the open sets of $N_{\boldsymbol{\nu}}$ have 
the form $\tilde{U}/\tilde{T}$, where $\tilde{U}\subseteq A^\vee_{\boldsymbol{\nu}}$ is open and
$\tilde{T}$-invariant.
The previous argument shows that each $\tilde{T}$-orbit in $\tilde{A}^\vee_{\boldsymbol{\nu}}$
intersects $X_{\boldsymbol{\nu}}$ in a single 
$T$-orbit.
One can see from this (and the definition 
of $\tilde{A}^\vee_{\boldsymbol{\nu}}$) that there is a
bijection between the family of $\tilde{T}$-invariant open sets $\tilde{U}$ in $\tilde{A}^\vee_{\boldsymbol{\nu}}$ and the family of
$T$-invariant open sets $U$ in $X_{\boldsymbol{\nu}}$ given by
$\tilde{U}\mapsto U:=\tilde{U}\cap X_{\boldsymbol{\nu}}$,
with inverse
$
U\mapsto \tilde{U}:=\tilde{T}\cdot U$.

Given any such $\tilde{U}$, we have
$
\psi^{-1}(\tilde{U}/\tilde{T})=
U/T\subseteq N_{\boldsymbol{\nu}}'$,
implying that $\psi$ is continuous. Similarly, given any such $U$ we have
$
\psi(U/T)=\tilde{U}/\tilde{T}
$,
implying that $\psi$ is open.
Hence $\psi$ is a homeomorphism.

To conclude that $\psi$ is an isomorphism of complex orbifolds,
it suffices to verify that its local expressions in uniformizing charts are biholomorphisms; actually, it suffices to do so for corresponding defining families 
in the sense of \cite{sat1} and \cite{sat2}
that cover $N_{\boldsymbol{\nu}}'$ and 
$N_{\boldsymbol{\nu}}$.
Let $F$ be a slice at $x$ for 
$\mu^{X_{\boldsymbol{\nu}}}$ at some $x\in X_{\boldsymbol{\nu}}$; by Proposition \ref{prop:slices TtildeT}, perhaps after shrinking $F$ if necessary, we may assume that
$F$ is also a slice at $x$ for $\tilde{\mu}^{\tilde{A}^\vee_{\boldsymbol{\nu}}}$. Hence 
$(F,J^F)$ is a
uniformizing chart of both $N_{\boldsymbol{\nu}}'$ and $N_{\boldsymbol{\nu}}$. 
By definition of $\psi$ and the previous considerations, the identity 
$\mathrm{id}_F:F\rightarrow F$ is a local representative map of $\psi$, and it is obviously biholomorphic $(F,J^F)\rightarrow
(F,J^F)$.
\end{proof}

The sheaf of holomorphic functions on $N_{\boldsymbol{\nu}}$ 
is defined equivalently by the $\tilde{T}$-invariant
holomorphic functions on 
$\tilde{A}^\vee_{\boldsymbol{\nu}}$
or the $T_{\ell}$-invariant 
holomorphic functions on the slices
$(F,J^F)$. 
Let us briefly clarify this point.

Since $(F,J^F)$ is generally not a complex submanifold of $(A^\vee,J')$, arbitrary
holomorphic functions on the saturation 
$\tilde{T}\cdot F$ needn't restrict
to holomorphic functions on $(F,J^F)$. However, this does happens 
if we restrict to invariant holomorphic functions.

\begin{defn}
\label{defn:inv hol fctns}
Suppose that $\ell\in A^\vee_{\boldsymbol{\nu}}$ and that $F\subseteq A^\vee_{\boldsymbol{\nu}}$ 
is a slice at $\ell$ for 
$\tilde{\mu}^{\tilde{A}^\vee_{\boldsymbol{\nu}}}$. 
Let us adopt the following notation.

\begin{enumerate}
\item $\mathcal{O}(F)$ is the ring of $J^F$-holomorphic
functions on $F$;
\item $\mathcal{O}(F)^{T_\ell}\subseteq \mathcal{O}(F)$ is the subring of $T_\ell$-invariant functions
in $\mathcal{O}(F)$;
\item $\mathcal{O}(\tilde{T}\cdot F)$ is the ring of $J'$-holomorphic
functions on the saturation of $F$ under
$\tilde{\mu}^{A^\vee}$;
\item $\mathcal{O}(\tilde{T}\cdot F)^{\tilde{T}}\subseteq \mathcal{O}(\tilde{T}\cdot F)$ is the subring
of $\tilde{\mu}^{\tilde{A}^\vee_{\boldsymbol{\nu}}}$-invariant functions.
\end{enumerate}

\end{defn}

Then we have the following, whose prooof will be omitted (see the argument for Proposition \ref{prop:restrictionCRolo}).

\begin{lem}
\label{lem:natural iso inv hol}
In the situation of Definition \ref{defn:inv hol fctns}, restriction yields an isomorphism 
$\mathcal{O}(\tilde{T}\cdot F)^{\tilde{T}}\rightarrow \mathcal{O}(F)^{T_\ell}$.
 
\end{lem}

\subsection{Holomorphic and CR functions 
on $\tilde{A}^\vee_{\boldsymbol{\nu}}$ and
$X_{\boldsymbol{\nu}}$}

$M_{\boldsymbol{\nu}}$ is a CR submanifold 
of $M$, and the maximal complex sub-bundle 
$\mathcal{H}(M_{\boldsymbol{\nu}})\subseteq TM_{\boldsymbol{\nu}}$ has complex dimension
$d+1-r$, and is as follows.
If $m\in M_{\boldsymbol{\nu}}$, 
$(\tilde{t}_{\boldsymbol{\nu}^\perp}^{r-1})_M(m)\subseteq T_mM_{\boldsymbol{\nu}}$
is a complex subspace of dimension $r-1$,
since
$\tilde{\gamma}^M$ is locally free at $m$.
Then
$$
\mathcal{H}(M_{\boldsymbol{\nu}})_m=
(\tilde{t}_{\boldsymbol{\nu}^\perp}^{r-1})_M(m)^{\perp_{h_m}},
$$
where $h_m=g_m-\imath\,\omega_m$ is the Hermitian product on
$T_mM$ associated to the K\"{a}hler metric.

Similarly, $X_{\boldsymbol{\nu}}$ is a CR submanifold 
of $A^\vee$. The maximal complex sub-bundle 
$\mathcal{H}(X_{\boldsymbol{\nu}})\subset TX_{\boldsymbol{\nu}}$ is as follows.
If $x\in X_{\boldsymbol{\nu}}$ and $m=\pi(x)$, then
\begin{equation}
\label{eqnHXnudef}
\mathcal{H}(X_{\boldsymbol{\nu}})_x=\mathcal{H}(M_{\boldsymbol{\nu}})_m^\sharp.
\end{equation}

\begin{defn}
\label{defn:equivariant CRholo}
Let be given $\boldsymbol{\lambda}\in \mathbb{Z}^r$.

For any $\tilde{T}$-invariant open subset
$\tilde{U}\subseteq \tilde{A}^\vee_{\boldsymbol{\nu}}$, let 
$\mathcal{O}(\tilde{U})_{\boldsymbol{\lambda}}$ be the 
ring of holomorphic functions  
$\tilde{S}:
\tilde{U}\rightarrow \mathbb{C}$ such that 
\begin{equation}
\label{eqn:olo equivariante}
\tilde{S}\left(\tilde{\mu}^{A^\vee}
_{\tilde{\mathbf{t}}^{-1}}(\ell)\right)
=\tilde{\chi}_{\boldsymbol{\lambda}}(\tilde{\mathbf{t}})\,
\tilde{S}(\ell)\qquad
(\tilde{\mathbf{t}}\in \tilde{T},\, \ell\in \tilde{U}). 
\end{equation}

For any $T$-invariant open subset
$U\subseteq X_{\boldsymbol{\nu}}$, let
let 
$\mathcal{CR}(U)_{\boldsymbol{\lambda}}$
be the ring of CR functions on $U$ satisfying
\begin{equation}
\label{eqn:CR equivariante}
S\left(\mu^c_{\mathbf{t}^{-1}}(x)\right)=\chi_{\boldsymbol{\lambda}}(\mathbf{t})\,S(x)\qquad 
(\mathbf{t}\in T,\,x\in X_{\boldsymbol{\nu}}).
\end{equation}

\end{defn}

\begin{prop}
\label{prop:restrictionCRolo}
With notation in Definition
\ref{defn:equivariant CRholo}, 
suppose that 
$U=\tilde{U}\cap X_{\boldsymbol{\nu}}$.
Then restriction yields a ring isomorphism 
$\mathcal{O}(\tilde{U})_{\boldsymbol{\lambda}}\rightarrow
\mathcal{CR}(U)_{\boldsymbol{\lambda}}$.

\end{prop}

\begin{cor}
\label{cor:iso holo CR global}
Restriction yields an isomorphism
$\mathcal{O}(\tilde{A}^\vee_{\boldsymbol{\nu}})_{\boldsymbol{\lambda}}\rightarrow
\mathcal{CR}(X_{\boldsymbol{\nu}})_{\boldsymbol{\lambda}}$.
\end{cor}

\begin{proof}
[Proof of Proposition \ref{prop:restrictionCRolo}]
Clearly if $\tilde{S}\in 
\mathcal{O}(\tilde{U})_{\boldsymbol{\lambda}}$
then 
$S:=\left.\tilde{S}\right|_U\in 
\mathcal{CR}(U)_{\boldsymbol{\lambda}}$.
Thus the ring homomorphim in the statement
is well-defined and obviously injective.

To prove surjectivity,
suppose conversely that 
$S\in 
\mathcal{CR}(U)_{\boldsymbol{\lambda}}$.
Let us define 
$\tilde{S}:\tilde{U}=\tilde{T}\cdot U\rightarrow\mathbb{C}$ by setting
\begin{equation}
\label{eqn:holo extension}
\tilde{S}\left(\tilde{\mu}^{A^\vee_0}_{\tilde{\mathbf{t}}}(x)\right)=
\tilde{\chi}_{-\boldsymbol{\lambda}}(\tilde{\mathbf{t}})\,S(x)
=
\tilde{\chi}_{\boldsymbol{\lambda}}
(\tilde{\mathbf{t}})^{-1}\,S(x)\qquad
(\tilde{\mathbf{t}}\in \tilde{T},\quad x\in X_{\boldsymbol{\nu}}). 
\end{equation}
To verify that (\ref{eqn:holo extension}) is well-defined, suppose that 
$\tilde{\mu}^{A^\vee_0}_{\tilde{\mathbf{t}}_1}(x_1)=\tilde{\mu}^{A^\vee_0}_{\tilde{\mathbf{t}}_2}(x_2)$ with $\tilde{\mathbf{t}}_j\in \tilde{T}$
and $x_j\in U$.
By the argument in the proof of Proposition \ref{prop:psi bijective}, $\tilde{\mathbf{t}}_2^{-1}\,\tilde{\mathbf{t}}_1\in T$. Therefore
$$
S(x_2)=\chi_{-\boldsymbol{\lambda}}\left(\tilde{\mathbf{t}}_2^{-1}\,  
\tilde{\mathbf{t}}_1  \right)\,
S(x_1)=
\tilde{\chi}_{-\boldsymbol{\lambda}}\left(\tilde{\mathbf{t}}_2\right)^{-1}\,  
\tilde{\chi}_{-\boldsymbol{\lambda}}\left(\tilde{\mathbf{t}}_1  \right)\,S(x_1).
$$

By construction, $\tilde{S}$ satisfies (\ref{eqn:olo equivariante}) and restricts to
$S$ on $U$.
To prove that $S\mapsto\tilde{S}$
inverts restriction it remains to verify that
$\tilde{S}$ is holomorphic, i.e. that
$\mathrm{d}_\ell \tilde{S}$
is $\mathbb{C}$-linear for any $\ell\in \tilde{U}$. 

By Corollary \ref{cor:tildeXnu vs tildeAnu}, the map
$$
\mathcal{F} :(\tilde{\mathbf{t}},x)\in\tilde{T}\times X_{\boldsymbol{\nu}}\mapsto
\tilde{\mu}^{A^\vee_0}_{\tilde{\mathbf{t}}}(x)\in \tilde{A}^\vee_{\boldsymbol{\nu}}
$$
is surjective. In fact, $\mathcal{F} $ exhibits
$\tilde{A}^\vee_{\boldsymbol{\nu}}$ as the 
quotient of 
$\tilde{T}\times X_{\boldsymbol{\nu}}$
by the free action of $T$ given by
\begin{equation}
\label{eqn:free T action}
\mathbf{t}\cdot \left(\tilde{\mathbf{t}},\,x\right)
:=\left(\tilde{\mathbf{t}}\,\mathbf{t}^{-1},\,\mu^X_{\mathbf{t}}(x)\right).
\end{equation}
Furthermore, $\mathcal{F} (\tilde{T}\times U)=\tilde{U}$.

For every $\tilde{\mathbf{t}}\in \tilde{T}$ let us set 
$$
X_{\boldsymbol{\nu}}^{\tilde{\mathbf{t}}}:=
\mathcal{F}\left(\left\{\tilde{\mathbf{t}}\right\}\times
X_{\boldsymbol{\nu}}\right)=
\tilde{\mu}^{A^\vee}_{\tilde{\mathbf{t}}}\left(
X_{\boldsymbol{\nu}}\right).
$$
Again by the proof of 
Proposition \ref{prop:psi bijective}
we have 
$X_{\boldsymbol{\nu}}^{\tilde{\mathbf{t}}_1}=
X_{\boldsymbol{\nu}}^{\tilde{\mathbf{t}}_2}$
if $\tilde{\mathbf{t}}_1^{-1}\,\tilde{\mathbf{t}}_2\in 
T$, and 
$X_{\boldsymbol{\nu}}^{\tilde{\mathbf{t}}_1}\cap
X_{\boldsymbol{\nu}}^{\tilde{\mathbf{t}}_2}=\emptyset$
otherwise.
Clearly $X_{\boldsymbol{\nu}}^{\tilde{\mathbf{t}}}$
is a CR submanifold of $A^\vee$, 
and its CR
bundle $\mathcal{H}(X_{\boldsymbol{\nu}}^{\tilde{\mathbf{t}}})$ is as follows. If
 $\ell=\mathcal{F}(\tilde{\mathbf{t}},x)\in X_{\boldsymbol{\nu}}^{\tilde{\mathbf{t}}}$,
then 
 $$
 \mathcal{H}(X_{\boldsymbol{\nu}}^{\tilde{\mathbf{t}}})_\ell
 =\mathrm{d}_x
 \tilde{\mu}^{A^\vee}_{\tilde{\mathbf{t}}}
 \big(
\mathcal{H}(X_{\boldsymbol{\nu}})_x\big).
$$

If $\tilde{\mathbf{t}}\in \tilde{T}$, let us
identify $T_{\tilde{\mathbf{t}}}\tilde{T}\cong
\tilde{•\mathfrak{t}}$ in the standard manner.
For $(\tilde{\mathbf{t}},\,x)\in 
\tilde{T}\times X_{\boldsymbol{\nu}}$, 
let us consider the vector subspace
$$
\mathcal{K}(\tilde{\mathbf{t}},x):=
\tilde{\mathfrak{t}}\times \mathcal{H}(X_{\boldsymbol{\nu}})_x
\subseteq T_{\tilde{\mathbf{t}}}\tilde{T}\times 
T_xX_{\boldsymbol{\nu}}
\cong T_{(\tilde{\mathbf{t}},x)}
(\tilde{T}\times X_{\boldsymbol{\nu}}).
$$

The distribution 
$\mathcal{K}\subseteq T(\tilde{T}\times X_{\boldsymbol{\nu}})$ is invariant under 
(\ref{eqn:free T action}) and
is naturally a complex vector bundle;
furthermore,
$\mathrm{d}\mathcal{F}$
yields an isomorphism
of complex vector bundles 
$\mathcal{K}\rightarrow \mathcal{F}^*(T\tilde{A}^\vee_{\boldsymbol{\nu}})$.
More explicitly,
if $\ell=\mathcal{F} (\tilde{\mathbf{t}},x)$
then 
\begin{equation}
\label{eqn:dFtxiso}
\left.\mathrm{d}_{(\tilde{\mathbf{t}},x)}\mathcal{F} \right|
_{\mathcal{K}(\tilde{t},x)}:\tilde{\mathfrak{t}}\times \mathcal{H}(X_{\boldsymbol{\nu}})_x\rightarrow
\tilde{\mathfrak{t}}_{A^\vee}(\ell)
\oplus \mathcal{H}(X_{\boldsymbol{\nu}}^{\tilde{\mathbf{t}}})_\ell =T_\ell A^\vee
\end{equation}
is an isomorphism of complex vector spaces,
respecting the direct sum decompositions on both sides.

Given $S\in \mathcal{CR}(U)_{\boldsymbol{\lambda}}$, let us consider
the complex function $\hat{S}$ on
$\tilde{T}\times X_{\boldsymbol{\nu}}$
given by
\begin{equation}
\label{eqn:defnShat}
\hat{S}\left(\tilde{\mathbf{t}},\,x\right):=
\tilde{\chi}_{\boldsymbol{\nu}}
(\tilde{\mathbf{t}})^{-1}\,S(x).
\end{equation}
Then 
$\hat{S}=\tilde{S}\circ \mathcal{F} $.

Let us assume that $\ell=\mathcal{F} (\tilde{\mathbf{t}},x)$.
We have
$$
\left.\mathrm{d}_{(\tilde{\mathbf{t}},x)}\hat{S}\right|
_{\mathcal{H}(\tilde{\mathbf{t}},x)}
=\mathrm{d}_{\ell}\tilde{S}\circ 
\left.\mathrm{d}_{(\tilde{\mathbf{t}},x)}\mathcal{F} 
\right|
_{\mathcal{K}(\tilde{\mathbf{t}},x)}:
\mathcal{K}(\tilde{\mathbf{t}},x)\rightarrow\mathbb{C}.
$$
Hence to prove that $\mathrm{d}_{\ell}\tilde{S}:
T_\ell A^\vee_{\boldsymbol{\nu}}
\rightarrow \mathbb{C}$
is $\mathbb{C}$-linear it suffices to show
that 
$\mathrm{d}_{(\tilde{\mathbf{t}},x)}\hat{S}$
is $\mathbb{C}$-linear on
$\mathcal{K}_{(\tilde{\mathbf{t}},x)}
=\tilde{•\mathfrak{t}}\times \mathcal{H}(X_{\boldsymbol{\nu}})_x$; 
to do so, in turn it is sufficient to verify 
$\mathbb{C}$-linearity on each summand
$\tilde{\mathfrak{t}}$ and 
$\mathcal{H}(X_{\boldsymbol{\nu}})_x$
separately.
This follows from
(\ref{eqn:defnShat}), since $\tilde{\chi}_{\boldsymbol{{\nu}}}$ 
is holomorphic (implying $\mathbb{C}$-linearity on the first summand) and $S$ is CR 
(implying $\mathbb{C}$-linearity on the second summand).

\end{proof}

\subsection{The orbifold line bundles $B_{\boldsymbol{\nu}}$ and $B'_{\boldsymbol{\nu}}$}

We have seen that the restrictions of $\mu^X$ to $X_{\boldsymbol{\nu}}$ 
and of $\tilde{\mu}^{A^\vee}$ to
$\tilde{A}^\vee_{\boldsymbol{\nu}}$
are locally free, effective and proper actions of $T$ and $\tilde{T}$, respectively, and that the corresponding
quotients $N_{\boldsymbol{\nu}}':=X_{\boldsymbol{\nu}}/T$ and 
$N_{\boldsymbol{\nu}}:=\tilde{A}^\vee_{\boldsymbol{\nu}}/\tilde{T}$ are naturally isomorphic
complex orbifolds. Furthermore, the 
projections
$p_{\boldsymbol{\nu}}':X_{\boldsymbol{\nu}}
\rightarrow N_{\boldsymbol{\nu}}'$
and $p_{\boldsymbol{\nu}}:
\tilde{A}^\nu_{\boldsymbol{\nu}}
\rightarrow N_{\boldsymbol{\nu}}$
are principal $V$-bundles with structure
group $T$ and $\tilde{T}$, respectively.

Associated to the characters $\chi_{\boldsymbol{\nu}}:T\rightarrow S^1$ 
and $\tilde{\chi}_{\boldsymbol{\nu}}:\tilde{T}\rightarrow \mathbb{C}^*$,
we have $1$-dimensional representations of $T$ and $\tilde{T}$, respectively; 
we shall
denote either one by $\mathbb{C}_{\boldsymbol{\nu}}$. 
The product actions
$\mu^{X\times \mathbb{C}_{\boldsymbol{\nu}}}$
and $\tilde{\mu}^{A^\vee\times \mathbb{C}_{\boldsymbol{\nu}}}$ are therefore also locally free, effective and proper
on $X_{\boldsymbol{\nu}}\times \mathbb{C}_{\boldsymbol{\nu}}$ 
and $\tilde{A}^\vee_{\boldsymbol{\nu}}\times \mathbb{C}_{\boldsymbol{\nu}}$
respectively. 
Hence the quotients
$B_{\boldsymbol{\nu}}':=X_{\boldsymbol{\nu}}\times_{T}\mathbb{C}_{\boldsymbol{\nu}}$
and
$B_{\boldsymbol{\nu}}:=\tilde{A}^\vee_{\boldsymbol{\nu}}\times_{\tilde{T}}\mathbb{C}_{\boldsymbol{\nu}}$
are orbifold line bundles on
$N_{\boldsymbol{\nu}}'$ and
$N_{\boldsymbol{\nu}}$. Let us denote by 
\begin{equation}
\label{eqn:orbibundles}
P_{\boldsymbol{\nu}}':B_{\boldsymbol{\nu}}'\rightarrow N_{\boldsymbol{\nu}}'\quad
\text{and}\quad
P_{\boldsymbol{\nu}}:B_{\boldsymbol{\nu}}\rightarrow N_{\boldsymbol{\nu}}
\end{equation}
the respective projections.

\begin{lem}
\label{lem:fetta prodotto}
Suppose $x\in X_{\boldsymbol{\nu}}$
and let $F\subseteq X_{\boldsymbol{\nu}}$
be a slice for the restriction of $\mu^X$ to $X_{\boldsymbol{\nu}}$. Then
$F\times \mathbb{C}_{\boldsymbol{\nu}}$ is a slice
at $(x,0)$
for the 
restriction of $\mu^{X\times \mathbb{C}_{\boldsymbol{\nu}}}$
to $X_{\boldsymbol{\nu}}\times \mathbb{C}_{\boldsymbol{\nu}}$. 
The collection of all these slices yields a defining family
for $B_{\boldsymbol{\nu}}'$.

Similarly, suppose
$\ell\in \tilde{A}^\vee_{\boldsymbol{\nu}}$
and
let $F\subseteq \tilde{A}^\vee_{\boldsymbol{\nu}}$ be a slice at
$\ell$ for the restriction of 
$\tilde{\mu}^{A^\vee}$
to $\tilde{A}^\vee_{\boldsymbol{\nu}}$.
Then  
$F\times \mathbb{C}_{\boldsymbol{\nu}}$ is a slice
at $(\ell,0)$
for the 
restriction of $\tilde{•\mu}^{A^\vee\times \mathbb{C}_{\boldsymbol{\nu}}}$
to $\tilde{A}^\vee_{\boldsymbol{\nu}}\times \mathbb{C}_{\boldsymbol{\nu}}$. 
The collection of all these slices yields a defining family
for $B_{\boldsymbol{\nu}}$.

\end{lem}

\begin{proof}
[Proof of Lemma \ref{lem:fetta prodotto}]
Let us consider the former statement, the proof of the latter being similar.

Since $F\times \mathbb{C}_{\boldsymbol{\nu}}$ is transverse
to the $T$-orbits in $X_{\boldsymbol{\nu}}\times \mathbb{C}_{\boldsymbol{\nu}}$,
the map
$h:T\times (F\times \mathbb{C}_{\boldsymbol{\nu}})
\rightarrow X_{\boldsymbol{\nu}}\times \mathbb{C}_{\boldsymbol{\nu}}$
induced by the diagonal action is a local diffeomorphism onto the open 
saturation $T\cdot (F\times \mathbb{C}_{\boldsymbol{\nu}})\subseteq 
X_{\boldsymbol{\nu}}\times \mathbb{C}$.

Clearly we have the equality of stabilizers $T_{(x,0)}=T_x$.
Furthermore, suppose $(y,w)\in F\times \mathbb{C}_{\boldsymbol{\nu}}$, $\mathbf{t}\in T$.
Then 
$$\mu^{X\times \mathbb{C}_{\boldsymbol{\nu}}}_{\mathbf{t}}(y,w)=
\left(\mu^{X}_{\mathbf{t}}(y),
\chi_{\boldsymbol{\nu}}(\mathbf{t})\,w\right)
\in F\times \mathbb{C}_{\boldsymbol{\nu}}
$$
if and only if $\mu^X_{\mathbf{t}}(y)\in F$, that is,
if and only if $\mathbf{t}\in T_x$.
Hence $h$ descends to a diffeomorphism
$$
\overline{h}:
\big(T\times (F\times \mathbb{C}_{\boldsymbol{\nu}})\big)/T_x\rightarrow
T\cdot (F\times \mathbb{C}_{\boldsymbol{\nu}}),
$$
where $T_x$ acts antidiagonally on $T\times (F\times \mathbb{C}_{\boldsymbol{\nu}})$.

\end{proof}

\begin{cor}
\label{cor:local quotient}
$F\times \mathbb{C}_{\boldsymbol{\nu}}$ with the diagonal action of $T_x$
uniformizes
the open set $(F\times \mathbb{C}_{\boldsymbol{\nu}})/T_x\subseteq B_{\boldsymbol{\nu}'}$.
The collection of all these uniformizing charts is a defnining family for the 
orbifold line bundle $B'_{\boldsymbol{\nu}}$.
A similar statement holds fo $B_{\boldsymbol{\nu}}$.
\end{cor}

Given the complex structure $J^F$ on each $F$ (Lemma \ref{lem:unique JF}), we obtain a product
complex structure on $F\times \mathbb{C}_{\boldsymbol{\nu}}$. Hence both
$B_{\boldsymbol{\nu}}'$ and $B_{\boldsymbol{\nu}}$ are complex orbifolds
of complex dimension $d+2-r$.

Since any $T$-orbit in $X_{\boldsymbol{\nu}}\times\mathbb{C}$ is contained in a unique
$\tilde{T}$-orbit in $\tilde{A}^\vee_{\boldsymbol{\nu}}\times\mathbb{C}$, there is
a natural continuous map $\tilde{\psi}:B_{\boldsymbol{\nu}}'\rightarrow B_{\boldsymbol{\nu}}$.
The proof of Proposition \ref{prop:psi bijective} can be adapted to
yield the following:

\begin{prop}
\label{prop:tildepsi bijective}
$\tilde{\psi}$ is an isomorphism of complex orbifolds, and
$\psi\circ P'_{\boldsymbol{\nu}}=P_{\boldsymbol{\nu}}\circ \tilde{\psi}$.
\end{prop}

\subsection{The orbifold circle bundle $Y_{\boldsymbol{\nu}}$}

We need an alternative description of $B_{\boldsymbol{\nu}}'$.
Consider the intermediate quotient 
$$
Y_{\boldsymbol{\nu}}:=X_{\boldsymbol{\nu}}/T^{r-1}_{\boldsymbol{\nu}^\perp};
$$
then $Y_{\boldsymbol{\nu}}$ is  compact orbifold, of (real) dimension
$2\,(d+1-r)+1$, and the integrable and invariant CR structure on $X_{\boldsymbol{\nu}}$
descends to an integrable CR structure on $Y_{\boldsymbol{\nu}}$.
We shall denote by $\mathcal{H}(Y_{\boldsymbol{\nu}})$ the CR bundle of
$Y_{\boldsymbol{\nu}}$.

Let
$T^1_{\boldsymbol{\nu}}\leqslant T$ be the connected compact subgroup
of $T$ associated to the Lie subalgebra 
$\mathrm{span}(\imath\,\boldsymbol{\nu})\subseteq \mathfrak{t}$.
Given that $\boldsymbol{\nu}$ is coprime,
we have a Lie group isomorphism
\begin{equation}
\label{eqn:kappanu}
\kappa_{\boldsymbol{\nu}} : e^{\imath\,\vartheta}\in S^1\mapsto e^{\imath\,\vartheta\,\boldsymbol{\nu}}:=
\left(e^{\imath\,\vartheta\,\nu_1},\ldots,e^{\imath\,\vartheta\,\nu_r}\right)\in T^1_{\boldsymbol{\nu}}.
\end{equation}
Let us set 
$\overline{T}^1_{\boldsymbol{\nu}}:=T/T^{r-1}_{\boldsymbol{\nu}^\perp}\cong
T^1_{\boldsymbol{\nu}}/(T^1_{\boldsymbol{\nu}}\cap T^{r-1}_{\boldsymbol{\nu}^\perp})$.

Suppose $x\in X_{\boldsymbol{\nu}}$, and let
$F\subseteq X_{\boldsymbol{\nu}}$ be a slice at $x$ for the
restriction of $\gamma^X$ to $X_{\boldsymbol{\nu}}$.
We can view $T\times F$ as a uniformizing chart
for the smooth orbifold $X_{\boldsymbol{\nu}}$, with uniformized open set
$T\cdot F=(T\times F)/T_x$.
Then
$\overline{T}^1_{\boldsymbol{\nu}}\times F$ is a uniformizing chart for $Y_{\boldsymbol{\nu}}$,
covering the open set $(T\cdot F)/T^{r-1}_{\boldsymbol{\nu}^\perp}$.

Explicitly, $T_x$ act effectively on $\overline{T}^1_{\boldsymbol{\nu}}\times F$
by
$$
t_0\cdot \left(\overline{t},f\right):=\left(\overline{t}\,\overline{t}_0^{-1},\mu^X_{t_0}(f)\right),
$$
where for any $t\in T$ we have set
$\overline{t}=t\,T^{r-1}_{\boldsymbol{\nu}^\perp}\in \overline{T}^1_{\boldsymbol{\nu}}$.
%
%
Then the map 
$$
\gamma: (\overline{t},f)\in \overline{T}^1_{\boldsymbol{\nu}}\times 
F\mapsto T^{r-1}_{\boldsymbol{\nu}^\perp}\cdot \mu^X_t(f)\in 
(T\cdot F)/T^{r-1}_{\boldsymbol{\nu}^\perp}\subseteq Y_{\boldsymbol{\nu}}
$$
induces a homeomorphism
$(T\cdot F)/T^{r-1}_{\boldsymbol{\nu}^\perp}=(\overline{T}^1\times F)/T_\ell$.
Letting $F$ vary, we obtain a defining family for $Y_{\boldsymbol{\nu}}$.

Furthermore, $\overline{T}^1_{\boldsymbol{\nu}}$ acts effectively on 
$Y_{\boldsymbol{\nu}}$, and 
$N'_{\boldsymbol{\nu}}=Y_{\boldsymbol{\nu}}/\overline{T}^1_{\boldsymbol{\nu}}$;
let $\sigma_{\boldsymbol{\nu}}:Y_{\boldsymbol{\nu}}\rightarrow N'_{\boldsymbol{\nu}}$
be the projection.
For each slice $F\subseteq X_{\boldsymbol{\nu}}$, as above, the local representation of
$\sigma_{\boldsymbol{\nu}}$ is the projection $\overline{T}^1_{\boldsymbol{\nu}}\times F
\rightarrow F$. 
Thus $Y_{\boldsymbol{\nu}}$ is a principal $V$-bundle over $N'_{\boldsymbol{\nu}}$, with
structure group $\overline{T}^1_{\boldsymbol{\nu}}$.

Being trivial on $T^{r-1}_{\boldsymbol{\nu}^\perp}$,
 $\chi_{\boldsymbol{\nu}}$ descends to a character 
 $\chi'_{\boldsymbol{\nu}}:\overline{T}^1_{\boldsymbol{\nu}}\rightarrow S^1$.
 
\begin{lem}
\label{lem:chinu'iso}
Given that $\boldsymbol{\nu}$ is coprime,
$\chi'_{\boldsymbol{\nu}}$
is a Lie group isomorphism.
\end{lem}

\begin{proof}
[Proof of Lemma \ref{lem:chinu'iso}]
Since 
$\overline{T}^1_{\boldsymbol{\nu}}\cong 
T^1_{\boldsymbol{\nu}}/\left(T^1_{\boldsymbol{\nu}}\cap T^{r-1}_{\boldsymbol{\nu}^\perp}\right)$,
the statement is equivalent to the equality
\begin{equation}
\label{eqn:kernel chi nu}
\ker(\left.\chi_{\boldsymbol{\nu}}\right|_{T^1_{\boldsymbol{\nu}}})
=T^1_{\boldsymbol{\nu}}\cap T^{r-1}_{\boldsymbol{\nu}^\perp};
\end{equation}
since clearly $T^{r-1}_{\boldsymbol{\nu}^\perp}\subseteq \ker(\chi_{\boldsymbol{\nu}})$,
we need only prove that $\ker(\left.\chi_{\boldsymbol{\nu}}\right|_{T^1_{\boldsymbol{\nu}}})
\subseteq T^1_{\boldsymbol{\nu}}\cap T^{r-1}_{\boldsymbol{\nu}^\perp}$.

Since $\boldsymbol{\nu}$ is coprime, there exists 
$\mathbf{k}=\begin{pmatrix}
k_1&\cdots&k_r
\end{pmatrix} \in \mathbb{Z}^r$ such that
$\langle\boldsymbol{\nu},\mathbf{k}\rangle=\sum_{j=1}^r k_j\,\nu_j=1$. 

Let $\kappa_{\boldsymbol{\nu}}$ be as in (\ref{eqn:kappanu}). 
Then 
\begin{equation}
\label{eqn:compositionknu}
\chi_{\boldsymbol{\nu}}\circ \kappa_{\boldsymbol{\nu}}
\left( e^{\imath\,\vartheta}\right)  =
\chi_{\boldsymbol{\nu}}\left( e^{\imath\,\vartheta\,\boldsymbol{\nu}}  \right)
=e^{\imath\,\vartheta\,\|\boldsymbol{\nu}\|^2}  \quad \left(e^{\imath\,\vartheta}\in S^1\right).
\end{equation}

Hence if $e^{\imath\,\vartheta\,\boldsymbol{\nu}}\in \ker(\chi_{\boldsymbol{\nu}})$, then
we may assume $\vartheta=\vartheta_j:=2\,\pi\,j/\|\boldsymbol{\nu}\|^2$ for some $j=0,\ldots,\|\boldsymbol{\nu}\|^2-1$.
We have 
$$
\langle\vartheta_j\,\boldsymbol{\nu},\boldsymbol{\nu}\rangle =
\frac{2\,\pi\,j}{\|\boldsymbol{\nu}\|^2} \,\langle\,\boldsymbol{\nu},\boldsymbol{\nu}\rangle
=2\,\pi\,j =2\,\pi\,j \,\langle \mathbf{k},\boldsymbol{\nu}\rangle,
$$
so that 
$
\vartheta_j\,\boldsymbol{\nu}-2\,\pi\,j\,\mathbf{k}\in \boldsymbol{\nu}^\perp
$.
Thus
$$
e^{\imath\,\vartheta_j\,\boldsymbol{\nu}}=
e^{\imath\,[\vartheta_j\,\boldsymbol{\nu}-2\,\pi\,j\,\mathbf{k}]}
\in T^1_{\boldsymbol{\nu}}\cap T^{r-1}_{\boldsymbol{\nu}^\perp} .
$$

\end{proof}

Since $k_{\boldsymbol{\nu}}$ is an isomorphism, 
(\ref{eqn:compositionknu}) implies
the following.

\begin{cor}
\label{cor:cardinalityinters}
Assuming that $\boldsymbol{\nu}$ is coprime, 
$$
\left|T^{r-1}_{\boldsymbol{\nu}^\perp}\cap T^1_{\boldsymbol{\nu}}\right|=
\|\boldsymbol{\nu}\|^2.
$$
\end{cor}

Given the isomorphism
$\chi'_{\boldsymbol{\nu}}: \overline{T}^1_{\boldsymbol{\nu}}\cong S^1$, 
we shall view $Y_{\boldsymbol{\nu}}$ as a
principal $V$-bundle over $N_{\boldsymbol{\nu}}'$ with structure group $S^1$.
Let us denote by 
\begin{equation}
\label{eqn:defnrhoYnu}
\sigma^{Y_{\boldsymbol{\nu}}}:S^1\times Y_{\boldsymbol{\nu}}\rightarrow
Y_{\boldsymbol{\nu}}
\end{equation}
the corresponding action.

Let 
\begin{equation}
\label{eqn:defnQnu}
Q_{\boldsymbol{\nu}}:X_{\boldsymbol{\nu}}\rightarrow Y_{\boldsymbol{\nu}}
\end{equation}
be the projection. Then $U$ is a $T$-invariant open subset of $X_{\boldsymbol{\nu}}$
if and only if its image $Q_{\boldsymbol{\nu}}(U)$ is 
a $\overline{T}^1_{\boldsymbol{\nu}}\cong 
S^1$-invariant open subset of $Y_{\boldsymbol{\nu}}$.
It follows (recall the proof of Proposition \ref{prop:psi bijective})
that there is a bijective correspondence between 
$\tilde{T}$-invariant open subsets $\tilde{U}\subseteq\tilde{A}^\vee_{\boldsymbol{\nu}}$, 
$T$-invariant open subsets $U\subseteq X_{\boldsymbol{\nu}}$,
$S^1$-invariant open subsets $\overline{U}\subseteq Y_{\boldsymbol{\nu}}$, given by
$U;=\tilde{U}\cap X_{\boldsymbol{\nu}}$, $\overline{U}:=Q_{\boldsymbol{\nu}}(U)$.

The character $\chi'_{k\,\boldsymbol{\nu}}=(\chi_{\boldsymbol{\nu}}')^k:\overline{T}^1_{\boldsymbol{\nu}}
\rightarrow S^1$
corresponds to the endomorphism $\chi_k:g\in S^1\mapsto g^k\in S^1$. 
Let us denote by $\mathcal{CR}(Y_{\boldsymbol{\nu}})$ the collection of all
CR functions on $Y_{\boldsymbol{\nu}}$, and for any
$k\in \mathbb{Z}$ let us set
\begin{equation}
\label{eqn:kth isotype Ynu}
\mathcal{CR}(Y_{\boldsymbol{\nu}})_k=
\left\{f\in \mathcal{CR}(Y_{\boldsymbol{\nu}})\,:\,
f\circ \sigma^{Y_{\boldsymbol{\nu}}}_{e^{-\imath\,\theta}}=e^{\imath\,k\,\theta}\,f,\quad
\forall\,e^{\imath\,\theta}\in S^1\right\}.
\end{equation}

Using that the CR structure of 
$Y_{\boldsymbol{\nu}}$ is obtained by descending the invariant
CR structure of $X_{\boldsymbol{\nu}}$,
we can complement Proposition \ref{prop:restrictionCRolo} and Corollary
\ref{cor:iso holo CR global} by the following isomorphisms induced by pull-back:
\begin{equation}
\label{eqn:restrictionCRolo}
\mathcal{O}\big(\tilde{U}\big)_{k\,\boldsymbol{\nu}}\cong
\mathcal{CR}(U)_{k\,\boldsymbol{\nu}}
\cong\mathcal{CR}\left( \overline{U}  \right)_k.
\end{equation}

Letting $H^0(N_{\boldsymbol{\nu}},B_{k\,\boldsymbol{\nu}})$ denote the
space of holomorphic sections of the orbifold line bundle $B_{k\,\boldsymbol{\nu}}$,
we conclude that
\begin{equation}
\label{eqn:restrictionCRolosections}
H^0(N_{\boldsymbol{\nu}},B_{k\,\boldsymbol{\nu}})
\cong \mathcal{O}\big(\tilde{A}^\vee_{\boldsymbol{\nu}}\big)_{k\,\boldsymbol{\nu}}\cong
\mathcal{CR}(X_{\boldsymbol{\nu}})_{k\,\boldsymbol{\nu}}
\cong\mathcal{CR}\left( Y_{\boldsymbol{\nu}}  \right)_k.
\end{equation}

\subsection{The induced K\"{a}hler structure of $N_{\boldsymbol{\nu}}$}

We shall see that $P_{\boldsymbol{\nu}}:B_{\boldsymbol{\nu}}\rightarrow N_{\boldsymbol{\nu}}$ is a positive holomorphic $V$-line bundle.
In view of Propositions \ref{prop:psi bijective} and \ref{prop:tildepsi bijective} we may equivalently
consider $P'_{\boldsymbol{\nu}}:B'_{\boldsymbol{\nu}}\rightarrow N'_{\boldsymbol{\nu}}$.
With $\alpha$ as in (\ref{eqn:key properties alpha}), let 
$\alpha^{X_{\boldsymbol{\nu}}}:=\jmath_{\boldsymbol{\nu}}^*(\alpha)$, where
\begin{equation}
\label{eqn:defnjmathnu}
\jmath_{\boldsymbol{\nu}}:X_{\boldsymbol{\nu}}\hookrightarrow X
\end{equation}
is the inclusion.
Then $\alpha^{X_{\boldsymbol{\nu}}}$ is $T$-invariant, and by definition of 
$X_{\boldsymbol{\nu}}$ for any 
$\boldsymbol{\xi}\in \boldsymbol{\nu}^\perp$ we have
$$
\iota\big((\imath\,\boldsymbol{\xi})_{X_{\boldsymbol{\nu}}}\big)\,\alpha^{X_{\boldsymbol{\nu}}}
=\jmath_{\boldsymbol{\nu}}^*\big( \iota\big((\imath\,\boldsymbol{\xi})_{X}\big)\,\alpha   \big)=
-\left\langle \Phi,\imath\,\boldsymbol{\xi}\right\rangle \circ \jmath_{\boldsymbol{\nu}}= 0.
$$
Hence $\alpha^{X_{\boldsymbol{\nu}}}$ is the pull-back of an orbifold 1-form
$\alpha^{Y_{\boldsymbol{\nu}}}$ on $Y_{\boldsymbol{\nu}}$.  
Similarly, being $T$-invariant, $\Phi$ descends to a smooth function 
$\overline{\Phi} : Y_{\boldsymbol{\nu}} \rightarrow \mathfrak{t}^\vee$; hence
$\Phi^{\boldsymbol{\nu}}=\langle\Phi,\,\imath\,\boldsymbol{\nu}\rangle$
descends to a smooth function 
$\overline{\Phi}^{\boldsymbol{\nu}}:Y_{\boldsymbol{\nu}} \rightarrow \mathbb{R}$.

Clearly, 
\begin{equation}
\label{eqn:iotaYnu}
\iota\big((\imath\,\boldsymbol{\nu})_{Y_{\boldsymbol{\nu}}}\big)\,
\alpha^{Y_{\boldsymbol{\nu}}}=-\overline{\Phi}^{\boldsymbol{\nu}}.
\end{equation}
Let us define
$$
\beta_{\boldsymbol{\nu}}:=\frac{\|\boldsymbol{\nu}\|^2}{\overline{\Phi}^{\boldsymbol{\nu}}}\,
\alpha^{Y_{\boldsymbol{\nu}}},
$$
and let $-\delta^{Y_{\boldsymbol{\nu}}}\in \mathfrak{X}(Y_{\boldsymbol{\nu}})$
be the infintesimal generator of $\sigma^{Y_{\boldsymbol{\nu}}}$ in (\ref{eqn:defnrhoYnu}).
Thus by Corollary \ref{cor:cardinalityinters}
\begin{equation}
\label{eqn:generator sigma}
-\delta^{Y_{\boldsymbol{\nu}}}=\frac{1}{\|\boldsymbol{\nu}\|^2}\,(\imath\,\boldsymbol{\nu})_{Y_{\boldsymbol{\nu}}}.
\end{equation}
Given (\ref{eqn:generator sigma}) and (\ref{eqn:iotaYnu}),
we conclude the following.

\begin{cor}
\label{cor:betaYnuconnection}
$\beta_{\boldsymbol{\nu}}$ is $\sigma^{Y_{\boldsymbol{\nu}}}$-invariant, and
$\beta_{\boldsymbol{\nu}}(\delta^{Y_{\boldsymbol{\nu}}})=1$.
\end{cor}

Hence $\beta_{\boldsymbol{\nu}}$ is a connection 1-form for the principal
$V$-bundle $P'_{\boldsymbol{\nu}}:B'_{\boldsymbol{\nu}}\rightarrow N'_{\boldsymbol{\nu}}$.
Explicitly, 
\begin{eqnarray}
\label{eqn:dbetanuexpl}
\mathrm{d}\beta_{\boldsymbol{\nu}}=
\|\boldsymbol{\nu}\|^2\,\left[\frac{1}{\overline{\Phi}^{\boldsymbol{\nu}}}\,
\mathrm{d}\alpha^{Y_{•\boldsymbol{\nu}}}
-\frac{1}{(\overline{\Phi}^{\boldsymbol{\nu}})^2}\,
\mathrm{d}\overline{\Phi}^{\boldsymbol{\nu}}\wedge \alpha^{Y_{•\boldsymbol{\nu}}}
\right],
\end{eqnarray}
and one can also verify
that $\iota(\delta^{Y_{\boldsymbol{\nu}}})\,\mathrm{d}\beta_{\boldsymbol{\nu}}=0$ by direct inspection 
using (\ref{eqn:generator sigma}) and (\ref{eqn:dbetanuexpl}).
Furthermore, the kernel of $\beta_{\boldsymbol{\nu}}$ 
is the CR bundle of $Y_{\boldsymbol{\nu}}$:
$$
\ker (\beta_{\boldsymbol{\nu}})=\ker(\alpha^{Y_{\boldsymbol{\nu}}})
=\mathcal{H}(Y_{\boldsymbol{\nu}}).
$$
Hence we reach the following conclusion.
Let $\pi_{\boldsymbol{\nu}}:Y_{\boldsymbol{\nu}}\rightarrow N'_{\boldsymbol{\nu}}$
be the projection.

\begin{lem}
\label{lem:11formetanu}
There exists a $(1,1)$-form $\eta_{\boldsymbol{\nu}}'$ on $N_{\boldsymbol{\nu}}'$
such that $\mathrm{d}(\beta_{\boldsymbol{\nu}})
=2\,\pi_{\boldsymbol{\nu}}^*(\eta_{\boldsymbol{\nu}}')$.
\end{lem}


We shall denote by $\eta_{\boldsymbol{\nu}}$ the corresponding form on $N_{\boldsymbol{\nu}}$.
With the notation of Proposition \ref{prop:psi bijective}, we have the following.

\begin{prop}
\label{prop:orbifold kahler}
$(N'_{\boldsymbol{\nu}},J^{N_{\boldsymbol{\nu}}'}, \eta_{\boldsymbol{\nu}}')$ 
and $(N_{\boldsymbol{\nu}},J^{N_{\boldsymbol{\nu}}}, \eta_{\boldsymbol{\nu}})$ 
are isomorphic K\"{a}hler orbifolds.
In particular, $(Y_{\boldsymbol{\nu}},\,\beta_{\boldsymbol{\nu}})$ is a 
contact orbifold.
\end{prop}

\begin{proof}
[Proof of Proposition \ref{prop:orbifold kahler}]
It suffices to prove that $(N'_{\boldsymbol{\nu}},J^{N_{\boldsymbol{\nu}}'}, \eta_{\boldsymbol{\nu}}')$ is a K\"{a}hler orbifold, since the other
statements follow directly.

The uniformized tangent space of $Y_{\boldsymbol{\nu}}$ splits as the direct sum
$V(Y_{\boldsymbol{\nu}})\oplus H(Y_{\boldsymbol{\nu}})$, where
$V(Y_{\boldsymbol{\nu}})$ is the tangent space to the orbits of
$\sigma^{Y_{\boldsymbol{\nu}}}$. 
To check that $\eta_{\boldsymbol{\nu}}$ is K\"{a}hler,
it suffices therefore to verify that the restriction of
$\mathrm{d}\beta_{\boldsymbol{\nu}}$
to $H(Y_{\boldsymbol{\nu}})$ is compatible with the complex structure. 
In view of (\ref{eqn:dbetanuexpl}) and $T$-invariance, we need only check that the 
form
\begin{eqnarray}
\label{eqn:dbetanuexplup}
Q_{\boldsymbol{\nu}}^*(\mathrm{d}\beta_{\boldsymbol{\nu}})=
\|\boldsymbol{\nu}\|^2\,\jmath_{\boldsymbol{\nu}}^*\left(\frac{1}{\Phi^{\boldsymbol{\nu}}}\,
\mathrm{d}\alpha
-\frac{1}{(\Phi^{\boldsymbol{\nu}})^2}\,
\mathrm{d}\Phi^{\boldsymbol{\nu}}\wedge \alpha
\right),
\end{eqnarray}
where $Q_{\boldsymbol{\nu}}$ is as in (\ref{eqn:defnQnu}) and $\jmath_{\boldsymbol{\nu}}$
as in (\ref{eqn:defnjmathnu}), is compatible with the complex structure
of the CR bundle $\mathcal{H}(X_{\boldsymbol{\nu}})$.

Suppose $x\in X_{\boldsymbol{\nu}}$ and let $m:=\pi(x)\in M_{\boldsymbol{\nu}}$.
The general vector in $\mathcal{H}( X_{\boldsymbol{\nu}})_x$ has the form
$v^\sharp$ for some $v\in \mathcal{H}(M_{\boldsymbol{\nu}})_m$
(see (\ref{eqnHXnudef})), and then
$J'_x(v^\sharp)=J_m(v)^\sharp$. By 
(\ref{eqn:dbetanuexplup}), for any $v,w\in \mathcal{H}(M_{\boldsymbol{\nu}})_m$
\begin{eqnarray}
\label{eqn:dbetanuexplup1}
Q_{\boldsymbol{\nu}}^*(\mathrm{d}\beta_{\boldsymbol{\nu}})_x(v^\sharp,w^\sharp)&=&
\frac{\|\boldsymbol{\nu}\|^2}{\Phi^{\boldsymbol{\nu}}(m)}\,
\mathrm{d}_x\alpha(v^\sharp,w^\sharp)=
\frac{\|\boldsymbol{\nu}\|^2}{\Phi^{\boldsymbol{\nu}}(m)}\,
2\,\omega_m(v,w).
\end{eqnarray}
The statement follows, since $\Phi^{\boldsymbol{\nu}}(m)>0$ by definition of
$M_{\boldsymbol{\nu}}$, $\mathcal{H}_m(M_{\boldsymbol{\nu}})\subseteq T_mM$
is a complex subspace, and $\omega$ is K\"{a}hler.

\end{proof}

\begin{cor}
\label{cor:polarized kahler orbifold}
$(N_{\boldsymbol{\nu}},B_{\boldsymbol{\nu}})$ is polarized K\"{a}hler orbifold.
\end{cor}

Here notation is as in Chaper 4 of \cite{bg}. By the Kodaira-Baily Vanishing Theorem
(\cite{baily}, \cite{bg}), we obtain the following conclusion.

\begin{cor}
\label{cor:vanishing kb}
$H^i(N_{\boldsymbol{\nu}},B_{k\,\boldsymbol{\nu}})=0$, $\forall\,i>0,\, k\gg 0$.
\end{cor}

\subsection{An Hamiltonian circle action on $N_{\boldsymbol{\nu}}$}

The action $\rho^X:S^1\times X\rightarrow X$ with infinitesimal generator
$-\partial_\theta$ in (\ref{eqn:key properties alpha}) is the contact lift of the
trivial circle action on $M$ corresponding to the moment map $\Phi=1$
(recall (\ref{eqn:xiX})).
We shall see that $\rho^X$ determines a contact action
on $(Y_{\boldsymbol{\nu}},\,\beta_{\boldsymbol{\nu}})$ and
an holomorphic Hamiltonian action on $(N_{\boldsymbol{\nu}}',J^{N'_{\boldsymbol{\nu}}},2\,\eta_{\boldsymbol{\nu}}')$,
such that the former is the contact lift of the latter by
(the orbifold version of) the procedure in
(\ref{eqn:xiX}), when we regard 
$Y_{\boldsymbol{\nu}}$ as an orbifold circle bundle on $X_{\boldsymbol{\nu}}$.

Clearly, $\rho^X$ commutes with $\mu^X$. Hence $\rho^X$ 
leaves $X_{\boldsymbol{\nu}}$ invariant and determines a restricted
action $\rho^{X_{\boldsymbol{\nu}}}:S^1\times X_{\boldsymbol{\nu}}\rightarrow X_{\boldsymbol{\nu}}$.
For the same reason $\rho^{X_{\boldsymbol{\nu}}}$ passes to the quotients
$Y_{\boldsymbol{\nu}}$ and $N_{\boldsymbol{\nu}}$. In other words, 
$\rho^{X_{\boldsymbol{\nu}}}$ descends to actions 
$\rho^{Y_{\boldsymbol{\nu}}}:S^1\times Y_{\boldsymbol{\nu}}\rightarrow Y_{\boldsymbol{\nu}}$ and 
$\rho^{N'_{\boldsymbol{\nu}}}:S^1\times N'_{\boldsymbol{\nu}}\rightarrow N'_{\boldsymbol{\nu}}$,
so that the projections 
$Q_{\boldsymbol{\nu}}:X_{\boldsymbol{\nu}}\rightarrow Y_{\boldsymbol{\nu}}$
and $\pi_{\boldsymbol{\nu}}:Y_{\boldsymbol{\nu}}\rightarrow N'_{\boldsymbol{\nu}}$
are equivariant.

In particular, if $-\partial_\theta^{X_{\boldsymbol{\nu}}}$ is the restriction of
$-\partial_\theta$ to $X_{\boldsymbol{\nu}}$, $-\partial_\theta^{Y_{\boldsymbol{\nu}}}$ 
is the infinitesimal generator of $\rho^{Y_{\boldsymbol{\nu}}}$, and 
$-\partial_\theta^{N_{\boldsymbol{\nu}}}$ 
is the infinitesimal generator of $\rho^{N_{\boldsymbol{\nu}}}$, then 
$\partial_\theta^{X_{\boldsymbol{\nu}}}$ and $\partial_\theta^{Y_{\boldsymbol{\nu}}}$
are $Q_{\boldsymbol{\nu}}$-related, and similarly $\partial_\theta^{Y_{\boldsymbol{\nu}}}$ and $\partial_\theta^{N_{\boldsymbol{\nu}}}$
are $\pi_{\boldsymbol{\nu}}$-related.

\begin{lem}
\label{lem:rhoNnuhamiltonian}
$\rho^{N'_{\boldsymbol{\nu}}}$ is Hamiltonian on 
$(N'_{\boldsymbol{\nu}},2\,\eta'_{\boldsymbol{\nu}})$, with moment map
$\|\boldsymbol{\nu}\|^2/\overline{\Phi}^{\boldsymbol{\nu}}+c$, for any $c\in \mathbb{R}$.
\end{lem}

\begin{proof}
[Proof of Lemma \ref{lem:rhoNnuhamiltonian}]
By $T$-invariance of all terms involved, 
and the previous remark about the correlations of the generating vector fields,
we need only prove that
$$
-\iota\left(\partial_\theta^{X_{\boldsymbol{\nu}}}\right)\,Q_{\boldsymbol{\nu}}^*(\mathrm{d}\beta_{\boldsymbol{\nu}})=\mathrm{d}\left(\|\boldsymbol{\nu}\|^2/\Phi^{\boldsymbol{\nu}}\circ \jmath_{\boldsymbol{\nu}}\right),
$$
where $\jmath_{\boldsymbol{\nu}}$ is as in (\ref{eqn:defnjmathnu})
and $Q_{\boldsymbol{\nu}}^*(\mathrm{d}\beta_{\boldsymbol{\nu}})$ as in (\ref{eqn:dbetanuexplup}).
We have on a neighborhood of $X_{\boldsymbol{\nu}}$:
$$
-\iota\left(\partial_\theta\right)\,\|\boldsymbol{\nu}\|^2\,\left(\frac{1}{\Phi^{\boldsymbol{\nu}}}\,
\mathrm{d}\alpha
-\frac{1}{(\Phi^{\boldsymbol{\nu}})^2}\,
\mathrm{d}\Phi^{\boldsymbol{\nu}}\wedge \alpha
\right)=-\,\frac{\|\boldsymbol{\nu}\|^2}{(\Phi^{\boldsymbol{\nu}})^2}\,
\mathrm{d}\Phi^{\boldsymbol{\nu}}=\mathrm{d}\left(\frac{\|\boldsymbol{\nu}\|^2}{\Phi^{\boldsymbol{\nu}}}\right),
$$
establishing the claim.
\end{proof}

Thus $-\partial_\theta^{N'_{\boldsymbol{\nu}}}$
is a Hamiltonian vector field on $(N'_{\boldsymbol{\nu}},\,2\,\eta'_{\boldsymbol{\nu}})$, and
every choice of $c\in \mathbb{R}$ in Lemma \ref{lem:rhoNnuhamiltonian}
determines a contact lift $-\widetilde{•\partial_\theta^{N'_{\boldsymbol{\nu}}}}$
(implicitly depending on $c$) to $(Y_{\boldsymbol{\nu}},\,\beta_{\boldsymbol{\nu}})$ 
of $-\partial_\theta^{N'_{\boldsymbol{\nu}}}$, as in
(\ref{eqn:xiX}). 

Here $Y_{\boldsymbol{\nu}}$ plays the role of
$X$, $\beta_{\boldsymbol{\nu}}$ the role of $\alpha$, and 
$-\partial_\theta^{N'_{\boldsymbol{\nu}}}$ the one of 
$\xi_M$.
The role of $-\partial_\theta$ (the infinitesimal generator of $\rho^X$) is played 
by $-\delta^{Y_{\boldsymbol{\nu}}}$
(the infinitesimal generator of $\sigma^{Y_{\boldsymbol{\nu}}}$).

We need to determine the \lq correct choice\rq\, of $c$ that determines 
$\rho^{Y_{\boldsymbol{\nu}}}$ as the contact lift of 
$\rho^{N'_{\boldsymbol{\nu}}}$.

\begin{lem}
\label{lem:c=0lift}
We have 
$\widetilde{•\partial_\theta^{N'_{\boldsymbol{\nu}}}}=
\partial_\theta^{Y_{\boldsymbol{\nu}}}$ if and  only if $c=0$.

\end{lem}

Given an (orbifold) vector field $V$ on $N_{\boldsymbol{\nu}}$, we shall denote by
$V^\natural\in \mathfrak{X}(Y_{\boldsymbol{\nu}})$ its horizontal lift to
$Y_{\boldsymbol{\nu}}$ with respect to $\beta_{\boldsymbol{\nu}}$.

\begin{proof}
[Proof of Lemma \ref{lem:c=0lift}]
On $X_{\boldsymbol{\nu}}$ we have by (\ref{eqn:xiX})
\begin{equation}
\label{eqn:relation on Xu}
\frac{1}{\|\boldsymbol{\nu}\|^2}\,\boldsymbol{\nu}_{X_{\boldsymbol{\nu}}}
=\frac{1}{\|\boldsymbol{\nu}\|^2}\,\boldsymbol{\nu}_{M_{\boldsymbol{\nu}}}^\sharp
-\frac{\Phi^{\boldsymbol{\nu}}}{\|\boldsymbol{\nu}\|^2}\,\partial_\theta^{X_{\boldsymbol{\nu}}}.
\end{equation}
Here $\boldsymbol{\nu}_{M_{\boldsymbol{\nu}}}$ is the restriction of 
$\boldsymbol{\nu}_M$ to $M_{\boldsymbol{\nu}}$ (a vector field on $M_{\boldsymbol{\nu}}$), and
$\boldsymbol{\nu}_{M_{\boldsymbol{\nu}}}^\sharp$ is its horizontal lift to 
$X_{\boldsymbol{\nu}}$.

Given that $\rho^X$ and $\mu^X$ commute, $[\boldsymbol{\nu}_X,\partial_\theta]=0$ on $X$;
this implies  
$[\boldsymbol{\nu}_M^\sharp,\boldsymbol{\nu}_X]=[\boldsymbol{\nu}_M^\sharp,\partial_\theta]=0$.
Furthermore, one has $[\boldsymbol{\nu}_X,\boldsymbol{\gamma}_X]=0$
for every $\boldsymbol{\gamma}\in \mathfrak{t}$, and this implies also
$[\boldsymbol{\nu}_M^\sharp,\boldsymbol{\gamma}_X]=0$.

Being horizontal and $T^{r-1}_{\boldsymbol{\nu}^\perp}$-invariant, 
$\boldsymbol{\nu}_{M_{\boldsymbol{\nu}}}^\sharp/\Phi^{\boldsymbol{\nu}}$ is 
$\pi_{\boldsymbol{\nu}}$-related to a horizontal vector field on $Y_{\boldsymbol{\nu}}$;
the latter is $\sigma^{Y^{\boldsymbol{\nu}}}$-invariant by the above, and therefore it is the 
horizontal lift $-\upsilon^\natural$ to $Y_{\boldsymbol{\nu}}$ 
of a vector field $-\upsilon$ on $N_{\boldsymbol{\nu}}$.

Multiplying both sides of (\ref{eqn:relation on Xu}) by 
$\|\boldsymbol{\nu}\|^2/\Phi^{\boldsymbol{\nu}}$ and then pushing down  
to $Y_{\boldsymbol{\nu}}$ we obtain

\begin{equation}
\label{eqn:relation on Yu}
\upsilon^\natural-\frac{\|\boldsymbol{\nu}\|^2}{\overline{•\Phi^{\boldsymbol{\nu}}}}\,\delta^{Y_{\boldsymbol{\nu}}}=-\frac{1}{\overline{\Phi^{\boldsymbol{\nu}}}}\,\boldsymbol{\nu}_{M_{\boldsymbol{\nu}}}^\sharp-\frac{\|\boldsymbol{\nu}\|^2}{\overline{•\Phi^{\boldsymbol{\nu}}}}\,\delta^{Y_{\boldsymbol{\nu}}}
=
-\partial_\theta^{Y_{\boldsymbol{\nu}}},
\end{equation}
and pushing down to $N_{\boldsymbol{\nu}}$ this yields
\begin{equation}
\label{eqn:relation on Nu}
\upsilon=
-\partial_\theta^{N_{\boldsymbol{\nu}}}.
\end{equation}
In view of Lemma \ref{lem:rhoNnuhamiltonian}, (\ref{eqn:relation on Nu}) 
implies that $\upsilon$ is the Hamiltonian vector field
on $(N'_{\boldsymbol{\nu}},\,2\,\eta'_{\boldsymbol{\nu}})$ of $\|\boldsymbol{\nu}\|^2/\overline{\Phi^{\boldsymbol{\nu}}}+c$;
then (\ref{eqn:relation on Yu}) implies that 
$-\partial_\theta^{Y_{\boldsymbol{\nu}}}$ is its
contact lift corresponding to $c=0$. It is clear that any other choice of $c$
yields a different lift.

\end{proof}

In the following, we shall identify the pairs 
$(N'_{\boldsymbol{\nu}},B'_{k\,\boldsymbol{\nu}})
\cong (N_{\boldsymbol{\nu}},B_{k\,\boldsymbol{\nu}})$.

\subsection{The Fourier decomposition of 
$\mathcal{O}(\tilde{A}^\vee_{\boldsymbol{\nu}})_{k\,\boldsymbol{\nu}}$}

Consider the holomorphic action 
$$
\rho^{A^\vee_0}:(e^{\imath\,\theta},\,\ell)\in S^1\times A^\vee_0\rightarrow
e^{-\imath\,\theta}\,\ell\in A^\vee_0.
$$
Thus $\rho^{A^\vee_0}$ extends $\rho^X$.
Similarly, let $\mu^{A^\vee_0}:T\times A^\vee_0\rightarrow
A^\vee_0$ be the holomorphic action extending $\mu^X$.
Clearly, $\rho^{A^\vee_0}$ and
$\mu^{A^\vee_0}$ commute.

The dense open subset 
$\tilde{A}^\vee_{\boldsymbol{\nu}}\subseteq A^\vee_{0}$
is invariant under both $\rho^{A^\vee_0}$ and $\mu^{A^\vee_0}$, which
therefore restrict to commuting holomorphic actions 
$\rho^{\tilde{A}^\vee_{\boldsymbol{\nu}}}$ and 
$\mu^{\tilde{A}^\vee_{\boldsymbol{\nu}}}$ on
$\tilde{A}^\vee_{\boldsymbol{\nu}}$.

Therefore, $\rho^{\tilde{A}^\vee_{\boldsymbol{\nu}}}$ and
$\mu^{\tilde{A}^\vee_{\boldsymbol{\nu}}}$ determine commuting
representations $\hat{\rho}^{\tilde{A}^\vee_{\boldsymbol{\nu}}}$ of $S^1$
and $\hat{\mu}^{\tilde{A}^\vee_{\boldsymbol{\nu}}}$ of $T$
on the space $\mathcal{O}(\tilde{A}^\vee_{\boldsymbol{\nu}})$
of holomorphic functions on $\tilde{A}^\vee_{\boldsymbol{\nu}}$,
given by
\begin{equation*}
\hat{\rho}^{\tilde{A}^\vee_{\boldsymbol{\nu}}}_{e^{\imath\theta}}(s):=
s\circ \rho^{\tilde{A}^\vee_{\boldsymbol{\nu}}}_{e^{-\imath\theta}},\quad
\hat{\mu}^{\tilde{A}^\vee_{\boldsymbol{\nu}}}_{\mathbf{t}}(s):=
s\circ \mu^{\tilde{A}^\vee_{\boldsymbol{\nu}}}_{\mathbf{t}^{-1}}\qquad \big(s\in 
\mathcal{O}(\tilde{A}^\vee_{\boldsymbol{\nu}}),\,
e^{\imath\theta}\in S^1,\,\mathbf{t}\in T.
\big).
\end{equation*}

For every $l\in \mathbb{Z}$ and $\boldsymbol{\lambda}\in \mathbb{Z}^r$, 
let $\mathcal{O}(\tilde{A}^\vee_{\boldsymbol{\nu}})_l$
and $\mathcal{O}(\tilde{A}^\vee_{\boldsymbol{\nu}})_{\boldsymbol{\lambda}}$
be the $l$-th and 
$\boldsymbol{\lambda}$-th
isotypical components of $\mathcal{O}(\tilde{A}^\vee_{\boldsymbol{\nu}})$, respectively, 
for $\hat{\rho}^{\tilde{A}^\vee_{\boldsymbol{\nu}}}$ and $\hat{\mu}^{\tilde{A}^\vee_{\boldsymbol{\nu}}}$,
respectively.
Hence 
$\hat{\rho}^{\tilde{A}^\vee_{\boldsymbol{\nu}}}$ restricts to a subrepresentation
on $\mathcal{O}(\tilde{A}^\vee_{\boldsymbol{\nu}})_{\boldsymbol{\lambda}}$. 
In particular, for every $k=1,2,\ldots$ the vector space
$\mathcal{O}(\tilde{A}^\vee_{\boldsymbol{\nu}})_{k\,\boldsymbol{\nu}}$
is finite dimensional by (\ref{eqn:restrictionCRolosections}), and
we have an $S^1\times T$-equivariant decomposition 
\begin{equation}
\label{eqn:equivariant S1T}
\mathcal{O}(\tilde{A}^\vee_{\boldsymbol{\nu}})_{k\,\boldsymbol{\nu}}=
\bigoplus_{l\in \mathbb{Z}}\mathcal{O}(\tilde{A}^\vee_{\boldsymbol{\nu}})_{k\,\boldsymbol{\nu},\,l},
\end{equation}
where $\mathcal{O}(\tilde{A}^\vee_{\boldsymbol{\nu}})_{k\,\boldsymbol{\nu},\,l}=
\mathcal{O}(\tilde{A}^\vee_{\boldsymbol{\nu}})_{k\,\boldsymbol{\nu}}\cap \mathcal{O}(\tilde{A}^\vee_{\boldsymbol{\nu}})_{l}$.
Since
the isomorphisms in (\ref{eqn:restrictionCRolosections}) are by construction
$S^1$-equivariant, (\ref{eqn:equivariant S1T}) may be interpreted in terms 
of $H^0(N_{\boldsymbol{\nu}},B_{k\,\boldsymbol{\nu}})$:
\begin{equation}
\label{eqn:splitting H^0S1}
H^0(N_{\boldsymbol{\nu}},B_{k\,\boldsymbol{\nu}})
=\bigoplus_{l\in \mathbb{Z}}H^0(N_{\boldsymbol{\nu}},B_{k\,\boldsymbol{\nu}})_l.
\end{equation}

\begin{lem}
\label{lem:kgg0vanishing}
If $k\gg 0$, 
$H^0(N_{\boldsymbol{\nu}},B_{k\,\boldsymbol{\nu}})_l=0$ for all $l\le 0$.

\end{lem}

\begin{proof}
[Proof of Lemma \ref{lem:kgg0vanishing}]
In the terminology of \cite{ms},
the datum of the Hamiltonian action
$\rho^{N_{\boldsymbol{\nu}}}$, with moment map 
$\|\boldsymbol{\nu}\|^2/\overline{\Phi^{\boldsymbol{\nu}}}$,
makes $B_{k\,\boldsymbol{\nu}}$ into a prequantum $S^1$-equivariant
orbibundle, hence into a moment line bundle.
By Corollary 2.11 of \cite{ms}, and given that 
$\|\boldsymbol{\nu}\|^2/\overline{\Phi^{\boldsymbol{\nu}}}>0$,
we conclude that the Fourier decomposition of 
$\mathrm{RR}(N_{\boldsymbol{\nu}},B_{k\,\boldsymbol{\nu}})$ 
(viewed as a virtual character of $S^1$)
has the form
\begin{equation}
\label{eqn:RRpositivecontr}
\mathrm{RR}(N_{\boldsymbol{\nu}},B_{k\,\boldsymbol{\nu}})=
\sum_{l>0}\mathrm{RR}(N_{\boldsymbol{\nu}},B_{k\,\boldsymbol{\nu}})_l\cdot\chi_l,
\end{equation}
where where $\chi_l(e^{\imath\,\theta})=e^{\imath\,l\,\theta}$.
In view of Corollary
\ref{cor:vanishing kb} this means that, as a representation of $S^1$,
\begin{equation}
\label{eqn:splitting H^0poscontr}
H^0(N_{\boldsymbol{\nu}},B_{k\,\boldsymbol{\nu}})
=\bigoplus_{l>0}H^0(N_{\boldsymbol{\nu}},B_{k\,\boldsymbol{\nu}})_l\qquad 
\forall\,k\gg 0.
\end{equation}

\end{proof}

By the $S^1$-equivariance 
in (\ref{eqn:restrictionCRolosections}), we can now sharpen 
(\ref{eqn:equivariant S1T}) as follows.

\begin{cor}
\label{cor:positive-contribOA}
If $k\gg 0$, then
\begin{equation}
\label{eqn:equivariant S1Tpos}
\mathcal{O}(\tilde{A}^\vee_{\boldsymbol{\nu}})_{k\,\boldsymbol{\nu}}=
\bigoplus_{l>0}\mathcal{O}(\tilde{A}^\vee_{\boldsymbol{\nu}})_{k\,\boldsymbol{\nu},\,l}.
\end{equation}

\end{cor}

\section{Proof of Theorem \ref{thm:main}}

We can now give the proof of
Theorem \ref{thm:main}. First, however, let us consider the following
statement. 

\begin{lem}
\label{lem:isoHknuOknu}
For every $\boldsymbol{\lambda}\in \mathbb{Z}$, restriction yields an isomorphism
$\mathcal{O}(A^\vee_0)_{\boldsymbol{\lambda}}\cong H(X)^{\hat{\mu}}_{\boldsymbol{\lambda}}$.

\end{lem}

\begin{proof}
[Proof of Lemma \ref{lem:isoHknuOknu}]
Clearly, restriction yields a morphism 
$\zeta_{\boldsymbol{\lambda}}:\mathcal{O}(A^\vee_0)_{\boldsymbol{\lambda}}
\rightarrow H(X)^{\hat{\mu}}_{\boldsymbol{\lambda}}$.
If $f\in \mathcal{O}(A^\vee_0)$ is non-zero, then the locus where its differential
vanishes has real codimension $\ge 2$; if it vanishes on $X$, therefore, $f=0$.
Hence $\zeta_{\boldsymbol{\lambda}}$ is injective.

Since by assumption $\mathbf{0}\not\in \Phi(M)$, we have
$\dim H(X)^{\hat{\mu}}_{\boldsymbol{\lambda}}<+\infty$ for every $\boldsymbol{\lambda}•$.
Hence we have a finite direct sum 
$$
H(X)^{\hat{\mu}}_{\boldsymbol{\lambda}}
=\bigoplus_{l=a(\boldsymbol{\lambda})}^{b(\boldsymbol{\lambda})}H(X)^{\hat{\mu}}_{\boldsymbol{\lambda},l},
$$
where 
$$0\le a(\boldsymbol{\lambda})\le b(\boldsymbol{\lambda})<+\infty,\qquad
H(X)^{\hat{\mu}}_{\boldsymbol{\lambda},l}
:=H(X)^{\hat{\mu}}_{\boldsymbol{\lambda}}\cap 
H(X)_{l}.
$$
Hence, to verify that $\zeta_{\boldsymbol{\lambda}}$ is surjective, 
it suffices to show that any $s\in H(X)^{\hat{\mu}}_{\boldsymbol{\lambda},l}$
is the restriction of some $\tilde{s}\in \mathcal{O}(A^\vee_0)_{\boldsymbol{\lambda}}$.
Any $s\in H(X)^{\hat{\mu}}_{l}$ is the restriction of an holomorphic homogeneous function
of degree $l$, $\tilde{s}\in \mathcal{O}(A^\vee_0)_l$.
Since $\rho^{A^\vee_0}$ and $\gamma^{A^\vee_0}$ commute, one sees that
$\tilde{s}$ is in the $\boldsymbol{\lambda}$-th isotype for 
$T$, and therefore for $\tilde{T}$ as well. Hence $\zeta_{\boldsymbol{\lambda}}$ is surjective.

\end{proof}

\begin{proof}
[Proof of Theorem \ref{thm:main}]
By Lemma \ref{lem:isoHknuOknu}, for every $k=1,2,\ldots$
we have a natural equivariant injective linear map
\begin{equation}
\label{eqn:equivlinearinj}
F_{k\,\boldsymbol{\nu}}:=res_{k\,\boldsymbol{\nu}}\circ \zeta_{k\,\boldsymbol{\nu}}^{-1} :
H(X)^{\hat{\mu}}_{k\,\boldsymbol{\nu}}\rightarrow 
\mathcal{O}(\tilde{A}^\vee_{\boldsymbol{\nu}})_{k\,\boldsymbol{\nu}}\cong
H^0(N_{\boldsymbol{\nu}},B_{k\,\boldsymbol{\nu}}),
\end{equation}
where $res_{k\,\boldsymbol{\nu}}:
\mathcal{O}(A^\vee_{0})_{k\,\boldsymbol{\nu}}\rightarrow
\mathcal{O}(\tilde{A}^\vee_{\boldsymbol{\nu}})_{k\,\boldsymbol{\nu}}$
denotes restriction, and is obviously injective since $\tilde{A}^\vee_{\boldsymbol{\nu}}$
is open and dense in 
$A^\vee_{0}$; this proves the first statement of Theorem \ref{thm:main}.

To prove the second statement, it suffices to verify that 
$res_{k\,\boldsymbol{\nu}}$ is surjective for $k\gg 0$.
We have for some $c(k,\boldsymbol{\nu}),\,d(k,\boldsymbol{\nu})\in \mathbb{Z}$
with $c(k,\boldsymbol{\nu})\le d(k,\boldsymbol{\nu})$:
$$
\mathcal{O}(\tilde{A}^\vee_{\boldsymbol{\nu}})_{k\,\boldsymbol{\nu}}
=\bigoplus_{l=c(k,\boldsymbol{\nu})}^{d(k,\boldsymbol{\nu})}
\mathcal{O}(\tilde{A}^\vee_{\boldsymbol{\nu}})_{k\,\boldsymbol{\nu},l};
$$ 
hence 
$$
res_{k\,\boldsymbol{\nu}}=\bigoplus_{l=c(k,\boldsymbol{\nu})}^{d(k,\boldsymbol{\nu})}
res_{k\,\boldsymbol{\nu},l},
$$
where 
$$
res_{k\,\boldsymbol{\nu},l}:\mathcal{O}(A^\vee_{0})_{k\,\boldsymbol{\nu},l}\rightarrow
\mathcal{O}(\tilde{A}^\vee_{\boldsymbol{\nu}})_{k\,\boldsymbol{\nu},l}
$$
and we need to check that $res_{k\,\boldsymbol{\nu},l}$ is surjective for every 
$l=c(k,\boldsymbol{\nu}),\ldots,d(k,\boldsymbol{\nu})$
and $k\gg 0$.

By Corollary \ref{cor:positive-contribOA}, we may assume that $c(k,\boldsymbol{\nu})>0$.
Furthermore, 
by Lemma \ref{lem:XnuMnuinv}
$\tilde{A}^\vee_{\boldsymbol{\nu}}=(\pi')^{-1}(\tilde{M}_{\boldsymbol{\nu}})$ 
and therefore
$res_{k\,\boldsymbol{\nu},l}$ may canonically  reinterpreted in terms of 
the restriction of holomorphic sections:
\begin{equation}
\label{eqn:resnuklrexpressed}
\widetilde{res}_{k\,\boldsymbol{\nu},l}:
H^0\left( M,A^{\otimes l}   \right)_{k\,\boldsymbol{\nu}}\rightarrow
H^0\big( \tilde{M}_{\boldsymbol{\nu}},A^{\otimes l}   \big)_{k\,\boldsymbol{\nu}}.
\end{equation}

Hence we are reduced to proving that $\widetilde{res}_{k\,\boldsymbol{\nu},l}$
in (\ref{eqn:resnuklrexpressed}) is surjective for all $l>0$. 

Suppose $s\in H^0\left( M,A^{\otimes l}   \right)$. Then $s\in 
H^0\left( M,A^{\otimes l}   \right)_{k\,\boldsymbol{\nu}}$ if and only if the
following two conditions hold:
\begin{enumerate}
\item $s$ is $\gamma^X$-invariant, i.e., 
$s\in H^0(M,A^{\otimes l} )^{ T^{r-1}_{\boldsymbol{\nu}^\perp}}$;
\item for any $e^{\imath\,\vartheta}\in S^1$,
$$
\hat{\mu}_{e^{\imath\,\vartheta\,\boldsymbol{\nu}}}(s)=
e^{\imath\,k\,\|\boldsymbol{\nu}\|^2\,\vartheta}\,s.
$$
\end{enumerate}

In other words, we can identify $H^0\left( M,A^{\otimes l}   \right)_{k\,\boldsymbol{\nu}}$
with the $k\,\|\boldsymbol{\nu}\|^2$-isotypical component for the representation of
$T^1_{\boldsymbol{\nu}}\cong S^1$ on $H^0(M,A^{\otimes l} )^{ T^{r-1}_{\boldsymbol{\nu}^\perp}}$.
The same considerations apply to $H^0\big( \tilde{M}_{\boldsymbol{\nu}},A^{\otimes l}   \big)_{k\,\boldsymbol{\nu}}$.
We shall express this by writing
$$
H^0( M,A^{\otimes l}   )_{k\,\boldsymbol{\nu}}=
H^0(M,A^{\otimes l} )^{ T^{r-1}_{\boldsymbol{\nu}^\perp}}
_{k\,\|\boldsymbol{\nu}\|^2},\quad
H^0( \tilde{M}_{\boldsymbol{\nu}},A^{\otimes l}   )_{k\,\boldsymbol{\nu}}=
H^0(\tilde{M}_{\boldsymbol{\nu}},A^{\otimes l} )^{ T^{r-1}_{\boldsymbol{\nu}^\perp}}
_{k\,\|\boldsymbol{\nu}\|^2}.
$$

It is well-known that the restriction map
\begin{equation}
\label{eqn:resnuklrexpressed1}
f_{l,\boldsymbol{\nu}}:
H^0\left( M,A^{\otimes l}   \right)^{ T^{r-1}_{\boldsymbol{\nu}^\perp}}\rightarrow
H^0\big( \tilde{M}_{\boldsymbol{\nu}},A^{\otimes l}   \big)^{ T^{r-1}_{\boldsymbol{\nu}^\perp}}
\end{equation}
is an isomorphism (\S 5 of \cite{gs}, Theorem 2,18 of \cite{sj}), and it is clearly
$T^1_{\boldsymbol{\nu}}$-equivariant. The claim follows, since
$\widetilde{res}_{k\,\boldsymbol{\nu},l}$ is the restriction of
$f_{l,\boldsymbol{\nu}}$ to the
$k\,\|\boldsymbol{\nu}\|^2$-isotypical component, hence by equivariance it induces an 
isomorphism
$H^0(M,A^{\otimes l} )^{ T^{r-1}_{\boldsymbol{\nu}^\perp}}
_{k\,\|\boldsymbol{\nu}\|^2}\cong
H^0(\tilde{M}_{\boldsymbol{\nu}},A^{\otimes l} )^{ T^{r-1}_{\boldsymbol{\nu}^\perp}}
_{k\,\|\boldsymbol{\nu}\|^2}$.
\end{proof}

\end{document}